\pdfoutput=1
\documentclass[10pt,leqno]{amsart}
\usepackage{graphicx}
\usepackage[T1]{fontenc}

\usepackage{amsmath}
\usepackage{amssymb}
\usepackage{comment}
\usepackage{mathtools}  
\usepackage{xfrac}
\usepackage{float}
\usepackage{verbatim}
\usepackage{faktor}
\usepackage{soul}
\usepackage{lipsum}
\usepackage{amsthm}
\usepackage{titleps}
\usepackage{hyperref}
\usepackage{cleveref}
\usepackage{bbold}
\usepackage[dvipsnames]{xcolor}
\usepackage{tikz-cd}

\usepackage{wrapfig} 
\usepackage{pifont} 
\usepackage{enumitem} 
\usepackage{graphicx} 
\graphicspath{ {images/} }
\usepackage{pgf,tikz,pgfplots}
\pgfplotsset{compat=1.15}
\usepackage{mathrsfs}
\usetikzlibrary{arrows}
\usepackage{blindtext}
\usepackage{shapepar}
\usepackage{textcomp}
\usepackage{setspace}
\usepackage{multicol}

\newcommand\m[1]{\begin{bmatrix}#1\end{bmatrix}} 

\newcommand\sm[1]{\begin{bsmallmatrix}#1\end{bsmallmatrix}}

\newcommand{\Z}{\mathbb{Z}}
\newcommand{\Q}{\mathbb{Q}}
\newcommand{\F}{\mathbb{F}}

\newcommand{\PP}{\mathbb{P}}

\newcommand{\GL}{\textnormal{GL}}

\DeclareMathOperator{\Gal}{Gal}
\DeclareMathOperator{\Aut}{Aut}

\usepackage{fancyhdr,etoolbox}
\newtheorem{theorem}{Theorem}[]

\newtheorem{lemma}[theorem]{Lemma}

\newtheorem{proposition}[theorem]{Proposition}
\newtheorem{corollary}[theorem]{Corollary}

\theoremstyle{definition}

\newtheorem{definition}[theorem]{Definition}

\theoremstyle{remark}
\newtheorem{example}[theorem]{Example}
\newtheorem{remark}[theorem]{Remark}

\Crefname{theorem}{Theorem}{Theorems}
\Crefname{lemma}{Lemma}{Lemmas}
\Crefname{proposition}{Proposition}{Propositions}
\Crefname{corollary}{Corollary}{Corollaries}
\Crefname{conjecture}{Conjecture}{Conjectures}
\Crefname{definition}{Definition}{Definitions}
\Crefname{example}{Example}{Examples}
\Crefname{remark}{Remark}{Remarks}

\crefname{theorem}{theorem}{theorems}
\crefname{lemma}{lemma}{lemmas}
\crefname{proposition}{proposition}{propositions}
\crefname{corollary}{corollary}{corollaries}
\crefname{conjecture}{conjecture}{conjectures}
\crefname{definition}{definition}{definitions}
\crefname{example}{example}{examples}
\crefname{remark}{remark}{remarks}

\hypersetup{ colorlinks=true, linkcolor=black, filecolor=black, urlcolor=black }

\usepackage{lipsum}

\numberwithin{theorem}{section}

\begin{document}
\title{Number of $K$-rational points with given $j$-invariant on modular curves} 
\author[Initial Surname]{Ivan Novak}
\date{\today}
\address{Ivan Novak, University of Zagreb, Faculty of Science, Department of Mathematics, Bijeni\v{c}ka Cesta 30, 10000 Zagreb, Croatia}
\email{ivan.novak@math.hr}
\thanks{The author was financed by the Croatian
Science Foundation under the project no. IP-2022-10-5008 and by the project “Implementation of cutting-edge research and its application as part of the Scientific Center of Excellence for Quantum and Complex Systems, and Representations of Lie Algebras“, Grant No. PK.1.1.10.0004, co-financed by the European Union through the European Regional Development Fund - Competitiveness and Cohesion Programme 2021-2027.}

\maketitle

\let\thefootnote\relax

\begin{abstract}
In this article, we study how to compute the number of $K$-rational points with a given $j$-invariant on an arbitrary modular curve.  As an application, for each positive integer $n$, we determine the list of possible numbers of cyclic $n$-isogenies an elliptic curve over some number field can admit. Similarly, for an odd prime power $p^k$, we calculate the possible values for the number of points above some $j$-invariant on Cartan modular curves $X_{\mathrm s}(p^k)$, $X_{\mathrm{ns}}(p^k)$ and their normalizers. 

Combining known results about images of Galois representations of CM elliptic curves with our work, we also devise a simple algorithm to determine the number of rational CM points on any modular curve.

\end{abstract} 

\bigskip

\section{Introduction}

Let $H$ be a subgroup of $\GL_2(\Z/N\Z)$ and let $X_H$ be the corresponding modular curve. The aim of this article is to explore the possible cardinalities of the sets $$\{ x \in X_H(K) \mid j(x)=j_0\},$$ where $K$ is any number field and $j_0$ is any element of $K$. In our work, we do not attempt to pin down the number fields or elliptic curves for which each cardinality is achieved. In that sense, our work is purely group-theoretic.

In particular, setting $X_H$ to be the modular curve $X_0(n)$, we determine the list of possible values of the number of $n$-isogenies of an elliptic curve over a number field. Recall that an $n$-isogeny is a cyclic isogeny of degree $n$. 

Isogenies of elliptic curves over number fields are an interesting and well-studied topic. There is a lot of interest in classifying possible degrees of isogenies over number fields. Mazur \cite{Mazur} and Kenku \cite{kenku39, kenku169, kenku65, Kenku} have determined all possible degrees of cyclic isogenies of elliptic curves defined over $\Q$. This amounts to finding all positive integers $n$ for which there exists a non-cuspidal rational point on the modular curve $X_0(n)$. A generalization of this problem is the problem of classifying all points of some degree $d$ on $X_0(n)$. For $d>1$, this is much more difficult than for $d=1$. There is no complete classification of all degree $d$ points on all $X_0(n)$ for no value of $d>1$.  However, there has been a lot of recent progress in the case $d=2$. One of the hurdles which makes the problem of finding all quadratic points on $X_0(n)$ difficult is the fact that for infinitely many $n$ there exist quadratic points on $X_0(n)$ corresponding to elliptic curves with complex multiplication (CM). This complicates the problem, as it is usually easier to prove that a curve has no points than to prove that a non-empty list of points which were found is complete. When determining quadratic points on $X_0(n)$, it is often useful to count the points on $X_0(n)$ with a given CM $j$-invariant, or the total number of quadratic CM points. For example, in \cite{nikola}, Adžaga, Keller, Michaud-Jacobs, Najman, Ozman and Vukorepa developed various techniques for finding quadratic points on $X_0(n)$. In some cases, it is neccessary to count the number of CM points. The methods they use to count the CM points are often computationally difficult.

One may also fix a $j$-invariant $j_0$ and classify all possible degrees of points on $X_0(n)$ whose $j$-invariant equals $j_0$. This question was explored by Terao in \cite{kenji} for modular curves $X_0(n)$ and $X_1(n)$, and solved in the case of rational $j$-invariants. Similarly, in \cite{leastcmdeg}, Clark, Genao, Pollack and Saia consider the problem of determining the least degree of a CM point on modular curves $X_0(n), X_1(n)$ and $X_1(m,n)$.  

In this paper, we discuss questions of a similar type. As mentioned, we determine all possible values of the number of cyclic $n$-isogenies of an elliptic curve $E$ defined over some number field $K$. The case when $n$ is prime is classical. If a linear affine transformation of the projective line $\PP^1(\F_p)$ fixes three points (or even two points for $p=2$), it automatically fixes all $p+1$ points. The action of $\GL_2(\Z/p\Z)$ on subgroups of $(\Z/p\Z)^2$ of order $p$ is equivalent to the action of $\PP\GL_2(\F_p)$ on $\PP^1(\F_p)$. By considering the image of the Galois representation $\rho_{E,n}(G_K)$ of an elliptic curve as a subgroup of $\GL_2(\Z/p\Z)$, it follows that if an elliptic curve has three isogenies of degree $p$ defined over some field, it neccessarily has all $p+1$ isogenies of degree $p$ defined over that field. However, if one replaces $p$ by a composite integer $n$, this principle does not hold anymore. 

\begin{proposition}
    Let $p$ be an odd prime and $k$ a positive integer. Let $E$ be an elliptic curve over a number field $K$. The number of $p^k$-isogenies of $E$ which are defined over $K$ belongs to the set $$\{0\} \cup \{p^j \mid 0\leq j <k\}\cup \{2p^j \mid 0\leq j <k\}\cup \{p^k+p^{k-1}\}.$$ Furthermore, each of the numbers from the set is achieved for some number field $K$ and some elliptic curve $E/K$.  
\end{proposition}

\begin{proposition}
     Let $k>1$ be a positive integer. Let $E$ be an elliptic curve over a number field $K$. The number of $2^k$-isogenies of $E$ which are defined over $K$ belongs to the set $$\{0\} \cup \{2^m \mid m=1,\ldots, k\}\cup \{2^k+2^{k-1}\}.$$ Furthermore, each of the numbers from the set is achieved for some number field $K$ and some elliptic curve $E/K$.  
\end{proposition}
These results are proven in \Cref{sec3}. As we will see in \Cref{sec2}, if $K$ is a number field and $j_0\in K$, the number of $p^k$-isogenies defined over $K$ of an elliptic curve with $j$-invariant equal to $j_0$ is equal to the cardinality of the set $$\{ x \in X_0(p^k)(K) \mid j(x)=j_0\}.$$ 
The paper is organised in the following way. In \Cref{sec2}, we recall the moduli definition of modular curves $X_H$ (in fact, of $Y_H$) for any subgroup $H$ of $\GL_2(\Z/N\Z)$, write down explicitly the condition for a point on $X_H$ to be rational and discuss how to compute the number of rational points on $X_H$ above some $j$-invariant, given the image of the mod $N$ Galois representation. In \Cref{sec3}, we determine the possible number of $n$-isogenies of an elliptic curve over a number field. We also determine the possible number of $p^{k+1}$-isogenies that an elliptic curve with a fixed number of $p^k$-isogenies can have. In \Cref{sec4}, we determine the possible number of points above a $j$-invariant on Cartan modular curves $X_{\mathrm{s}}(p^k)$ and $X_{\mathrm{ns}}(p^k)$ and their normalisers, where $p^k$ is a power of an odd prime. In \Cref{sec5}, using known results on Galois representations of CM elliptic curves by Lozano-Robledo \cite{alvarocm}, we establish a simple algorithm for determining the number of rational CM points on an arbitrary modular curve.  A simple implementation of the algorithm is available on GitHub\cite{novak:cm-modular-curves-github}.\textsuperscript{1}\footnote{\textsuperscript{1}\url{https://github.com/inova3c/number-of-cm-points-on-modular-curves}}

There are several reasons why modular curves corresponding to Cartan subgroups and their normalizers are of interest. One reason is that Cartan subgroups and their normalizers appear as images of Galois representations of CM elliptic curves. This is explained in more detail in \Cref{sec5}. Another reason is Serre's uniformity conjecture. By Serre's open image theorem \cite{openimage}, for every elliptic curve $E/\Q$ there exists a constant $C_E>0$ such that the image of the mod $p$ Galois representation $\rho_{E,p}(G_\Q)$ is surjective for all primes $p>C_E$. Serre's uniformity conjecture asserts that $C_E$ does not depend on $E$.  In fact, it is widely believed that one can take $C_E=37$ for all $E$, i.e.\ that $\rho_{E,p}(G_\Q)$ is surjective for all $E/\Q$ and all $p>37$. While the uniformity conjecture remains a difficult open problem, there has been some partial progress. As of now, it has been proven (\cite{serrecheb, biluparent, bpr, fl}) that if $E/\Q$ is an elliptic curve and $p>37$ is a prime, then $\rho_{E,p}(G_\Q)$ is either the entire $\GL_2(\Z/p\Z)$ or conjugate to the normaliser of the non-split Cartan subgroup of $\GL_2(\Z/p\Z)$.  

\section*{Acknowledgements}

The author thanks Filip Najman, Matija Kazalicki and Enrique González Jiménez for helpful comments and corrections, to Maarten Derickx for suggesting that a result such as \Cref{ako 3 onda sve} should hold and to Jeremy Rouse for improvements to the algorithm in \Cref{sec5}.

\section{Number of $K$-rational points on a modular curve with a given $j$-invariant}\label{sec2}

We start by recalling general definitions regarding modular curves. Let $K$ be a number field and $j \in K$. Let $H \leq \GL_2(\Z/N\Z)$ be a subgroup. The non-cuspidal points on the modular curve $X_H$ are classes $[(E, \alpha)]$, where $E$ is an elliptic curve and $\alpha$ is an $H$-level structure on $E$.  More precisely, define the following equivalence relation on the set of all ordered bases for $E[N]$. We say $(P, Q)$ and $(P', Q')$ are equivalent mod $H$ if there is some $\sm{a & b \\ c & d} \in H$ such that $P'=aP+cQ$ and $Q'=bP+dQ$. An $H$-level structure $\alpha$ is an equivalence class with respect to this relation.

Furthermore, we say $(E, \alpha)$ and $(E', \alpha')$ are isomorphic if there is an isomorphism of elliptic curves  $\phi: E \to E'$ defined over $\overline \Q$ such that if $(P, Q)$ is a representative for the level structure $\alpha$, then $(\phi(P), \phi(Q))$ is a representative for the level structure $\alpha'$. A point $x$ on $X_H$ is an ordered pair $(E, \alpha)$, up to the mentioned isomorphism.

A point $x$ on $X_H(\overline \Q)$ is $K$-rational if $\sigma(x)=x$ for all $\sigma \in G_K$, where we define $\sigma([(E, \alpha)])$ to be $[(E^\sigma, \alpha^\sigma)]$, where if $\alpha$ is $(P,Q) \pmod{H}$, then $\alpha^\sigma$ is $(P^\sigma, Q^\sigma) \pmod{H}$. This means that $E$ and $E^\sigma$ are isomorphic over the algebraic closure $\overline K$, and the isomorphism sends $(P,Q) \pmod{H}$ to $(P^\sigma, Q^\sigma) \pmod H$. Since isomorphic elliptic curves have the same $j$-invariant, we have $j(E)^\sigma=j(E)$, hence $j(E)\in K$. 

This allows us to pick $E'$ isomorphic to $E$ such that $E'$ is defined over $K$, since all elliptic curves with the same $j$-invariant are isomorphic over the algebraic closure. Thus, it is not a loss of generality to assume that $E$ is defined over $K$ and $E=E^\sigma$. This means that the isomorphism is in fact an automorphism $\phi: E \to E$, and a rational point on $X_H$ can be represented as $(E, (P, Q) \pmod H)$, where $$(P^\sigma, Q^\sigma) \pmod H=(\phi(P), \phi(Q)) \pmod H.$$
In other words, for every $\sigma \in G_K$, there exists an automorphism $\phi$ of $E$ such that $(P^\sigma, Q^\sigma)$ equals $h(\phi(P),\phi(Q))$ for some $h \in H$. 

Denote by $\rho_{E,N}$ the mod $N$ Galois representation of an elliptic curve $E$. The images of $\rho_{E,N}(G_K)$ and $\Aut(E)$ inside $\GL_2(\Z/N\Z)$ depend on the choice of basis for $E[N]$. Let $R_{(P,Q)}$ and $A_{(P,Q)}$ respectively denote the images inside $\GL_2(\Z/N\Z)$ of $\rho_{E,N}(G_K)$ and $\Aut(E)$ with respect to the basis $(P,Q)$ for $E[N]$. The point $(E, (P,Q) \pmod H)$ is $K$-rational if and only for all $r \in R_{(P,Q)}$ there exists some $h\in H$ and $a \in A_{(P,Q)}$ such that $r=ha$. To determine the number of $K$-rational points on $X_H$ with a fixed $j$-invariant, we need to take any $E/K$ with that $j$-invariant and count the number of bases $(P,Q)$ such that $R_{(P,Q)}\subseteq H\cdot A_{(P,Q)}$, up to the action of $H$ and up to the action of $\Aut(E)$. 

\begin{lemma}\label{broj baza}
    Let $E/K$ be an elliptic curve and $(P_0,Q_0)$ a fixed basis for $E[N]$. Denote by $A$ the image of the action of $\Aut(E)$ on $E[N]$, written in basis $(P_0,Q_0)$. Let $(P,Q)=g(P_0, Q_0)$ be another basis for $E[N]$, where $g \in \GL_2(\Z/N\Z)$. The number of bases $(P', Q')$ such that $(E, (P,Q) \pmod H)$  and $(E, (P', Q') \pmod H)$ define the same point  on $X_H$ is equal to the cardinality of the double coset $HgA$.
\end{lemma}
\begin{proof}
    Any basis $(P', Q')$ can be uniquely written as $g(P_0, Q_0)$ for some $g \in \GL_2(\Z/N\Z)$. Then $(E, g_1(P_0,Q_0) \pmod H)$ and $(E, g_2(P_0, Q_0) \pmod H)$ define the same point on $X_H$ if and only if $g_1(P_0, Q_0) \pmod H=g_2(\phi(P), \phi(Q)) \pmod H$ for some $\phi \in \Aut(E)$, which happens if and only if there are some $h \in H$ and $\phi \in \Aut(E)$ such that $$g_1(P_0,Q_0)=h(g_2(a(P_0, Q_0))),$$ where $a$ denotes the matrix of action by the automorphism $\phi$ in basis $(P_0,Q_0)$. This is equivalent to $g_1=hg_2a$.
\end{proof}

\begin{proposition}
    Let $E/K$ be any elliptic curve with $j$-invariant $j_0$. Fix a basis $(P,Q)$ for $E[N]$. Denote by $R$ the image $\rho_{E,N}(G_K)$ in this basis and by $A$ the image of action of $\Aut(E)$ on $E[N]$ in this basis. The number of $K$-rational points $x$ on $X_H$ with $j(x)=j_0$ equals $$\sum_{g \in \GL_2(\Z/N\Z)} \frac{\mathbf{1}_{\{gR \subseteq HgA\}}}{\#(HgA)}.$$
\end{proposition}

\begin{proof}
    Note that the change of basis is equivalent to conjugating the groups $R$ and $A$ by some element $g \in \GL_2(\Z/N\Z)$. This means that the corresponding basis $(P',Q')$ yields a rational point on $X_H$ if and only if $$gRg^{-1} \subseteq H\cdot gAg^{-1}.$$
    By \Cref{broj baza}, there are $\#(HgA)$ bases which yield the same point on $X_H$, so we need to divide the corresponding summand by this number.
\end{proof}

\begin{remark}
Note that the condition $gR \subseteq HgA$ in fact only depends on the double coset of $g$, so the sum can be written to only span over double cosets, and equals $$\#\{HgA \in H\backslash G/A \mid gR\subseteq HgA\}.$$
Namely, if $gR \subseteq HgA$ and $h\in H$, then $hgR\subseteq hHgA=HhgA$.  Similarly, suppose that $gR \subseteq HgA$ and $a \in A$. Let $gar \in gaR$. Then note that $ra=a'r$ for some $a' \in A$, since if $\phi \in \Aut(E)$ and $\sigma \in G_K$, then $\sigma(\phi(P))=\phi^\sigma(\sigma(P))$, and $\phi^\sigma \in \Aut(E)$. In other words, each $r \in R$ normalizes $A$, so $AR=RA$, and $gar \in gRa'\subseteq HgAa'=HgA$.
\end{remark}

\begin{corollary}
    Let $j_0\neq 0, 1728$. With the notation from the previous proposition, the number of $K$-rational points $x$ on $X_H$ with $j(x)=j_0$ equals $$\frac{\#\{g \in \GL_2(\Z/N\Z) \mid gRg^{-1} \subseteq \pm H\}}{\#(\pm H)}.$$
\end{corollary}

Note that the number of points on $X_H$ and $X_{\pm H}$ is the same. The following well-known lemma explains this in a different way.


    

\begin{lemma}
    If $\rho_{E,N}(G_K)$ is a subgroup of $\pm H$ in some basis for $E[N]$ and $-I \not \in H$, there is a quadratic twist $E'$ of $E$ defined over $K$ such that $\rho_{E', N}(G_K)$ is conjugate to a subgroup of $H$. 
\end{lemma}
\begin{proof}
Take a basis in which the Galois image $\rho_{E,N}(G_K)$ is a subgroup of $\pm H$. Now, for each $\sigma \in G_K$, we either have $\rho(\sigma) \in H$ or $\rho(\sigma) \in -H$, which defines a homomorphism $G_K \to \{\pm 1\}$. The kernel of this homomorphism is an index $2$ subgroup, which is then equal to $G_{K(\sqrt d)}$ for some $d \in K$, and the mentioned homomorphism is the corresponding quadratic character $\chi_d$. Now take the quadratic twist of this elliptic curve by $d$. We then have $$\rho_{E', N}(\sigma)=\chi_d(\sigma)\rho_{E, N}(\sigma) \in H.$$
\end{proof}

If the point $x=(E, (P,Q) \pmod H)$ on $X_H$ is not defined over $K$ for $E/K$, one might be interested in determining the degree of $x$ over $K$. This is easy to calculate in terms of the image of the Galois representation in that basis.

\begin{lemma}
    Let $E/K$ be an elliptic curve, and fix a basis $(P_0,Q_0)$ for $E[N]$. Let $R\leq \GL_2(\Z/N\Z)$ be the Galois image $\rho_{E,N}(G_K)$ in that basis, and let $A$ be the image of the action of $\Aut(E)$ on $E[N]$ in that basis. Let $(P,Q)=g(P_0, Q_0)$ be a different basis, with $g \in \GL_2(\Z/N\Z)$.  The degree over $K$ of the point $(E, (P,Q)\pmod H)$ on the modular curve $X_H$ is equal to the index $[R:(R\cap(g^{-1}H gA))]$.
\end{lemma}
\begin{remark}
    Note that $g^{-1}HgA$ is not neccessarily a group if $A$ is larger than $\pm I$. However, in this setting, $R\cap(g^{-1}H gA)$ turns out to be a group.
\end{remark}
\begin{proof}
    The action of the Galois group $G_K$ on bases for $E[N]$ modulo the action of $H$ and $\Aut(E)$ is realised as the action of the group $R$ on double cosets $HgA$ for $g \in \GL_2(\Z/N\Z)$, and is given by $\sigma(HgA)=HgrA$, where $r\in \GL_2(\Z/N\Z)$ is the matrix determined by the action of $\sigma$ on $(P_0, Q_0)$. 
    
    The degree of $(E, g(P_0,Q_0) \pmod H)$ is equal to the size of the orbit of $HgA$ under the action of $R$, which is equal to the index of the stabilizer of $HgA$ inside $R$. The coset $HgA$ is stabilized if and only if $HgA=HgrA$, which is equivalent with $HgA=HgAr$, since $R$ normalizes $A$. 

    Now note that $HgA=HgAr$ holds for some $r \in R$ if and only if $grg^{-1} \in HgAg^{-1}$. Namely, if $grg^{-1}=hgag^{-1}$ for some $h \in H$, $a \in A$, then $r=g^{-1}hga$ and $$HgAr=HgrA=HhgaA=HgA.$$
    Conversely, if $HgA=HgrA$, then for each $hgra \in HgAr$, there are some $h' \in H$ and $a'\in A$ with $hgra=h'ga'$, which implies $r=g^{-1}h''ga''$ for some $h'' \in H$, $a \in A$. Thus, the stabilizer of $HgA$ is equal to $R \cap (g^{-1}HgA)$.
\end{proof}
\noindent More on the topic of degrees of points on modular curves can be found in Terao's article \cite{kenji}.



    






\section{Number of $p^k$-isogenies}\label{sec3}

Let $p$ be a prime and $k$ a positive integer. Let $E$ be an elliptic curve over some number field $K$. We want to determine the possible number of $p^k$-isogenies that $E$ can have over $K$. When $k=1$ and $p$ is odd, it is a well-known fact that $E$ has either $0,1, 2$ or $p+1$  isogenies of degree $p$ over $K$, and $E$ has either $0,1$ or $3$ isogenies of degree $2$ over $K$.

The number of $p^k$-isogenies is equal to the number of cyclic subgroups of order $p^k$ of $E[p^k]$ fixed by the Galois group $G_K$. Note that if we are free to choose both an elliptic curve and a number field, then we can completely reduce this problem to a question about subgroups of $\GL_2(\Z/p^k\Z)$ acting on $\Z/p^k\Z \times \Z/p^k\Z$. Namely, take any elliptic curve over a field $K$ such that $\rho_{E, p^k}(G_K)$ is surjective\textsuperscript{2}\footnote{\textsuperscript{2}For example, one may take the curve 37.a.1. from LMFDB. Its mod $p^k$ Galois image $\rho_{E,p^k}(G_\Q)$ is surjective for all prime powers $p^k$.}, and fix any basis for $E[p^k]$, so that $\rho_{E,p^k}(G_K)$ can be identified with $GL_2(\Z/N\Z)$.   If $H \leq \GL_2(\Z/p^k\Z)$ is a subgroup, $\rho^{-1}(H)$ is a subgroup of $G_K$ of finite index (it lies between $G_K$ and $G_{K(E[p^k])}$). Let $F$ be the fixed field of $\rho^{-1}(H)$. By Galois theory, $G_F=\rho^{-1}(H)$ and $\rho(G_F)=H$, so the image of the mod $p^k$ Galois representation of $E/F$ is equal to $H$.

The modular curve $X_0(p^k)$ parametrizes pairs $(E, C)$, where $E$ is an elliptic curve and $C$ is a cyclic subgroup of $E[p^k]$ of order $p^k$, up to isomorphism. The number of $p^k$-isogenies of an elliptic curve $E/K$ which are defined over $K$ is equal to the number of points on $X_0(p^k)$ whose $j$-invariant equals $j(E)$, unless $j=0$ or $j=1728$. In that case, two different $p^k$-isogenies can yield the same point on $X_0(p^k)$. For example, if $j=1728$ and $\zeta_3 \in K$, the automorphism $[\zeta_3]$ gives the equality of points $(E, C)$ and $(E, [\zeta_3]C)$ on the modular curve $X_0(p^k)$, despite the fact that the subgroup $C$ is not neccessarily equal to $[\zeta_3]C$.

In any case, if $R$ denotes the image of the mod $p^k$ Galois representation $\rho_{E,p^k}(G_K)$ in some basis, the number of $p^k$-isogenies over $K$ is equal to $$\frac{\#\{g \in \GL_2(\Z/p^k\Z) \mid gRg^{-1} \subseteq B_0(p^k)\} }{\#B_0(p^k)},$$
where $B_0(p^k)$ is the group of uppertriangular invertible $2\times 2$ matrices modulo $p^k$. 
Our goal is to determine all possible values that this expression can obtain, as $R$ varies over all subgroups  of $\GL_2(\Z/p^k\Z)$.

More generally, let $N$ be a positive integer and $H$ a subgroup of $\GL_2(\Z/N\Z)$. Denote by $S(H)$ the set of all values of the expression $$\frac{\#\{g \in \GL_2(\Z/N\Z) \mid gRg^{-1} \subseteq H\} }{\#H},$$
as $R$ ranges over all subgroups of $\GL_2(\Z/N\Z)$.
\begin{proposition}\label{neparno}
    Let $p$ be an odd prime and $k$ a positive integer. The set $S(B_0(p^k))$ is contained in the set $$\{0\} \cup \{p^j \mid 0\leq j <k\}\cup \{2p^j \mid 0\leq j <k\}\cup \{p^k+p^{k-1}\}.$$ 
\end{proposition}
\begin{proof}
    Let $R$ be a subgroup of $\GL_2(\Z/p^k\Z)$. If $R$ fixes no subgroups of $\Z/p^k\Z \times \Z/p^k\Z$ of order $p^k$, the number is $0$. Otherwise, $R$ is conjugate to a subgroup of $B_0(p^k)$, and it is not a loss of generality to assume that $R$ itself is a subgroup of $B_0(p^k)$.

    Now suppose that $gRg^{-1} \subseteq B_0(p^k)$. If $g=\sm{u & v \\ w & z}$, that means that for every $\sm{a & b \\ 0 & c}\in R$, the bottom-left entry of $\sm{u & v \\ w & z}\sm{a & b \\ 0 & c}\sm{u & v \\ w & z}^{-1}$ is equal to $0$. Explicitly calculating the bottom-left entry, we obtain the equality \begin{equation}\label{uvjet} w(z(a-c)-bw)=0, \ \text{for all } \sm{a & b \\ 0 & c}\in R.\end{equation}

    We now need to calculate the number of pairs $(w,z)$ for which this holds. Note that since $g \in \GL_2(\Z/p^k\Z)$, at least one among $w$ and $z$ is not divisible by $p$.

    Suppose that $w\not \equiv 0 \pmod{p}$. Dividing the equality \eqref{uvjet} by $w^2$ and setting $x=\frac{-z}{w}$, we obtain $$(c-a)x=b, \text{ for all } \sm{a & b \\0 & c } \in R.$$
    This is a system of equations in one variable over $\Z/p^k\Z$. We will call solutions of this system ``type A solutions''.

    Suppose that $w\equiv 0 \pmod p$. Then $z\not \equiv 0 \pmod p$. Dividing \cref{uvjet} by $z^2$ and setting $y=\frac{-w}{z}$, we obtain 
    $$(c-a)y=by^2, \text{ for all } \sm{a & b \\0 & c } \in R.$$
    Note that there is an additional constraint that $y\in p\Z/p^k\Z$. We will call solutions of this system ``type B solutions''.

    Now note that, for every solution of any type, there are $\#B_0(p^k)$ different matrices $\sm{u & v \\w & z}$ above it. Namely, if $x$ is a solution of type A, then we can choose $w$ to be any number coprime to $p$ in $\varphi(p^k)$ ways, we then have $z=-wx$, and we can choose $u$ and $v$ in $p^k \cdot \varphi(p^k)$ ways in such a way that $(u,v)$ and $(z,w)$ are coprime. Similarly, if $y$ is a solution of type B, we choose $z$ to be any number coprime to $p$, which uniquely determines $w$, and we get $p^k\cdot \varphi(p^k)$ ways to choose the entire matrix. Since $\#B_0(p^k)=p^k\cdot \varphi(p^k)^2$, this proves our claim.

    Let $A$ and $B$ respectively denote the number of solutions of type A and type B. We now need to determine all possible values of $A+B$. Denote by $\alpha$ the least possible value of $v_p(c-a)$ and by $\beta$ the least possible value of $v_p(b)$, for $\sm{a & b \\0 & c} \in R$.

   \noindent \textbf{Case 1.} There exists a solution of type A.  
   
   This means that there exists some $x\in \Z/p^k\Z$ such that $b=(c-a)x$ for all $\sm{a & b \\ 0 & c}\in R$. Note that $x$ is a solution of type A if and only if $x+p^{k-\alpha}$ is a solution, and there is a unique solution modulo $p^\alpha$. Thus, there are $p^\alpha$ solutions of type A.

    Since $b=(c-a)x$, any equation of the system for type B solutions can be written as $$(c-a)y=(c-a)xy^2.$$
    Note that any $y$ such that $v_p(y)\geq k-\alpha$ is a solution of type B. There are $p^\alpha$ such solutions, unless $\alpha=k$ since $y$ has to be divisible by $p$. In that case, there are $p^{k-1}$ such solutions.
    
    Suppose that some other $y$ with $v_p(y)<k-\alpha$ is a solution. Take $\sm{a & b \\0 & c}$ such that $v_p(c-a)=\alpha$. Then $v_p((c-a)y)<v_p((c-a)by^2)$, and we reach a contradiction. Thus, there are $p^\alpha$ solutions of type B, unless $\alpha=k$ in which case there are $p^{k-1}$ solutions. In total, there are $2p^\alpha$ fixed subgroups for $\alpha<k$ and $p^k+p^{k-1}$ fixed subgroups for $\alpha=k$.

\noindent \textbf{Case 2.} There are no solutions of type A. 

We say a solution $y$ of type B is \emph{proper} if $v_p(y)<k-\alpha$, so that $(c-a)y\neq 0$ for some $\sm{a & b \\0 & c} \in R$. We now split on two cases depending on whether there are proper solutions of type B.

\noindent \textbf{Case 2.(i)}  There are no proper solutions. 

In this case, all solutions $y$ satisfy $(c-a)y=0$ and $by^2=0$ for all $\sm{a & b \\0 & c}\in \rho$. This is equivalent to $v_p(y)\geq k-\alpha$ and $v_p(y)\geq \frac{k-\beta}{2}$. In total, there are $p^{\min\left(\alpha, \left\lfloor \frac{k+\beta}{2} \right\rfloor \right)}$ solutions.

\noindent \textbf{Case 2.(ii)} There exists a proper solution. 

We first prove that there exists an element $\sm{a & b \\0 & c}\in R$ with $v_p(c-a)=\alpha$ and $v_p(b)=\beta$. Suppose that there is no such element. Take $\sm{a_1 & b_1 \\
    0 & c_1}$ with $v_p(a_1-c_1)=\alpha$ and $\sm{a_2 & b_2 \\ 0 & c_2}$ with $v_p(b_2)=\beta$. We claim that the product $\sm{a & b \\ 0 & c}$ of these elements has $v_p(c-a)=\alpha$ and $v_p(b)=\beta$. Namely, we have $$\sm{a_1 & b_1 \\
     0 & c_1}\sm{a_2 & b_2 \\ 0 & c_2}=\sm{a_1a_2 & a_1b_2+b_1c_2 \\
     0 & c_1c_2}.$$
Then $$c-a=c_1c_2-a_1a_2=c_1(c_2-a_2)+a_2(c_1-a_1)$$ and $$b=a_1b_2+b_1c_2.$$ Since $a_i$ and $c_i$ are invertible, the claim follows from the fact that $v_p(x+y)=v_p(x)$ whenever $v_p(x)<v_p(y)$.

If $y$ is a proper solution, taking $\sm{a & b \\ 0 & c}$ which minimises $v_p(c-a)$ and $v_p(b)$, we obtain $$\alpha+v_p(y)=\beta+2v_p(y)<k,$$
which implies $v_p(y)=\alpha-\beta>0$. Thus, any proper solution has fixed $p$-adic valuation. This means that any other proper solution is of the form $ry$ for some invertible $r \in \Z/p^{k-\alpha+\beta}\Z$, and one has $$(c-a)ry=br^2y^2,$$ which is equivalent with $r\equiv 1 \pmod{p^{k-2\alpha+\beta}}$, which has $p^\alpha$ solutions in $\Z/p^{k-\alpha+\beta}\Z$. Thus, there are $p^\alpha$ proper solutions.

We have already deduced that the number of improper solutions is $p^{\min\left(\alpha, \left\lfloor \frac{k+\beta}{2} \right\rfloor \right)}$. Note that $\min\left(\alpha, \left\lfloor \frac{k+\beta}{2} \right\rfloor \right)<k$, as otherwise there would exist solutions of type A. We now need to prove $\alpha < \frac{k+\beta}{2}$. However, this is equivalent with $2\alpha-\beta<k$, which holds since for a proper solution $y$, one has $v_p((c-a)y)=2\alpha-\beta<k$. Thus, there are also $p^\alpha$ improper solutions, and $2p^\alpha$ solutions in total, which completes the proof.
\end{proof}

\begin{remark}
    Note that one can also arrive to the systems of equations of type A and type B directly. Namely, each cyclic subgroup of order $p^k$ either has a generator of the form $(x, 1)$ or of the form $(1,y)$ for $y$ divisible by $p$. In the first case, the subgroup  is fixed if and only if $\sm{a & b \\ 0 & c}\sm{x \\1}=\sm{s x \\ s},$ which is equivalent with $(c-a)x=b$, and this yields the system of type A. In the second case, the subgroup is fixed if and only if   $\sm{a & b \\ 0 & c}\sm{1 \\y}=\sm{s  \\ s y},$ which is equivalent with $(c-a)y=by^2$, and this yields the system of type B. 
\end{remark}

\begin{remark}
    In Case 1, when there do exist type A solutions, there exists some value $x$ such that $b=(c-a)x$ for every $\sm{a & b \\ 0 & c} \in R$. In that case, the group $R$ is in fact conjugate to the Cartan subgroup, i.e.\ the subgroup of diagonal matrices. To see this, note that $$\m{1 & -x \\ 0 & 1}\m{a & (c-a)x \\ 0 & c }\m{1 & x \\ 0 & 1}=\m{a & 0 \\ 0 & c}.$$
\end{remark}

\begin{remark}
    Throughout the proof, we have used $p$-adic valuation on $\Z/p^k\Z$, which may be problematic as $v_p(0)$ is not well defined. However, the property $v_p(xy)=v_p(x)+v_p(y)$ still holds whenever $xy\neq 0$ which is the only case in which we used it.
\end{remark}

\begin{proposition}
    Each number listed in \Cref{neparno} is achieved for some number field $K$ and some elliptic curve $E/K$.
\end{proposition}
\begin{proof} We keep the same notation from \Cref{neparno}. It suffices to find a subgroup $R$ which fixes a given number of subgroups, and then we can reconstruct a field and an elliptic curve as explained before. If $R$ is a subgroup consisting only of scalar matrices, then the number $p^k+p^{k-1}$ is achieved, and vice versa. Note that it is enough to check that all generators of $R$ fix a subgroup in order for the entire $R$ to fix it. 

If $R$ is not conjugate to a subgroup of the group of all uppertriangular matrices, for example if $R=\GL_2(\Z/p^k\Z)$, then $0$ is achieved. 

    Now fix $j\in \{0,\ldots, k-1\}$, and take $R$ to be the cyclic subgroup generated by $\sm{1 & p^j \\
    0 & p^j+1}$. The type A equation is $p^jx=p^j$, which has $p^j$ solutions. The type B equation is $p^jy=p^j y^2$, which also has $p^j$ solutions and we get $2p^j$ solutions in total. This settles the case of $2p^j$ for $j\in \{0, 1, \ldots, k-1\}$. 
    
    It remains to prove that each power $p^j$ for $j=0, \ldots, k-1$ can appear. To do this, we must ensure there are no solutions of type A and no proper solutions of type B, and $\min(\alpha, \lfloor (k+\beta)/2 \rfloor)$ should be $j$.

    For the case $j=0$, take $R$ to be the group of all uppertriangular matrices. Then $\alpha=0$ and there are no solutions of type A, no proper solutions of type B and only $y=0$ is an improper solution of type B.

    Now let $0<j<k$. Take $R$ to be the subgroup generated by $\sm{1 & 0 \\ 0 & 1+p^j}$ and $\sm{1 & p^j \\ 0 & 1+p^j}$. There are no type A solutions, and the type B equations are $p^j y=p^j y^2$ and $p^j y=0$. There are $p^j$ solutions. This completes the proof.
\end{proof}
The case $p=2$ is similar for the most part but the answer is slightly different, so we handle it separately.

\begin{proposition}
    We have $S(B_0(2))=\{0,1,3\}$. For $k>1$, $$S(B_0(2^k))=\{0\} \cup \{2^m \mid m=1,\ldots, k\}\cup \{2^k+2^{k-1}\}.$$ 
\end{proposition}

\begin{proof}
    The entire proof of \Cref{neparno} is also valid for $p=2$. The only difference is that one cannot have $\alpha=0$, as $c-a$ is always divisible by $2$. 

    If there exists a type A solution and $\alpha<k$, then there are $2^\alpha$ type A solutions, and $2^\alpha$ type B solutions, which gives $2^{\alpha+1}$ solutions in total.

    If there are no type A solutions and no proper solutions, then there are $2^{\min\left(\alpha, \left \lfloor \frac{k+\beta}{2}\right \rfloor\right)}$ solutions in total. If $k>1$, this number is one of $2^1, 2^2, \ldots, 2^k$.  If $k=1$, this number is one of $2^0, 2^1$. However, if $k=1$, then $\beta=1$ implies that there is a proper solution, so $2^1$ is impossible in that case.

    Finally, if there are no type A solutions and there is a proper type B solutions, then there are $2^{\alpha+1}$ solutions in total, unless $\alpha=k$ in which case there are $2^k+2^{k-1}$ solutions in total.

    Now let us comment why each number can be achieved. For $k=1$, this is well known. Let $k>1$. We wish to prove that each of $0,2^1,2^2, 2^3,\ldots, 2^k, 2^k+2^{k-1}$ can be achieved. The numbers $0$ and $2^k+2^{k-1}$ are obviously achieved by taking $R$ to be $\GL_2(\Z/2^k\Z)$ and the trivial group respectively.

    Taking $R$ to be the entire $B_0(2^k)$, the type A system is $(c-a)x=b$ for all odd $a,c$ and for all $b$, which has no solution. The type B system is $(c-a)y=by^2$ for all odd $a,c$ and for all $b$. In particular, for $b=0$, one has $(c-a)y=0$, which has $2$ solutions: $0$ and $2^{k-1}$. Both of these numbers are solutions of the entire system, so $2^1$ can be achieved.

    Now let $1\leq j<k$, and take $R$ to be generated by $\sm{1 & 2^j \\
    0 & 1+2^j}$. The type A equation is $2^jx=2^j$, which has $2^j$ solutions, and the type B equation is $2^jy=2^jy^2$, which also has $2^j$ solutions (remember, $y$ has to be divisible by $2$ to be a solution of type B). In total, we obtain $2^{j+1}$ solutions, as claimed.
\end{proof}

As a consequence, it is now easy to determine the possible number of $n$-isogenies of an elliptic curve for any $n$.

\begin{corollary}
    Let $n=\prod_{i=1}^r p_i^{e_i}$ be a positive integer. Then $S(B_0(n))$ is equal to the set of all products of the form $$x_1x_2 \ldots x_r, \ \text{ where } x_i \in S(B_0(p_i^{e_i})).$$
\end{corollary}
\begin{proof}
    This follows immediately from the fact that the number of $n$-isogenies is the product of the number of $p_i^{e_i}$-isogenies for $i=1,\ldots, r$ and the fact that for given subgroups $H_i\leq \GL_2(\Z/p_i^{e_i}\Z)$, there is a subgroup $H\leq \GL_2(\Z/N\Z)$ such that the reduction of $H$ modulo $p_i^{e_i}$ equals $H_i$. The second fact follows from the Chinese remainder theorem. 
\end{proof}
Now suppose that we are given an elliptic curve $E$ with a fixed number of $p^k$-isogenies. It is natural to ask how many $p^{k+1}$-isogenies $E$ can have. If we go through the proof of \Cref{neparno} more carefully, we can recover this information as well. 

For $m\in S(B_0(p^k))$, denote by $S_{k\to k+1}(m)$ the set of possible values for the number of $p^{k+1}$-isogenies of an elliptic curve whose number of $p^k$-isogenies is $m$. 

\begin{proposition}
    Let $p$ be an odd prime and $k$ a positive integer. Then \begin{itemize}
        \item $S_{k\to k+1}(p^k+p^{k-1})=\{0, p^k, 2p^k, p^{k+1}+p^k\}$, \ $S_{k\to k+1}(0)=\{0\}$.
        \item For $0\leq j<k$, \ $S_{k\to k+1}(2p^j) = \{0, p^j, 2p^j\}$.
        \item For $0\leq j<\frac{k-1}{2}$, \ $S_{k\to k+1}(p^j)=\{0, p^j\}$.
        \item For $j=\frac{k-1}{2}$, \ $S_{k \to k+1}(p^j)=\{0, p^j, p^{j+1}\}$.
        \item For $\frac{k}{2}\leq j<k$, \ $S_{k \to k+1}(p^j)=\{0, p^j, p^{j+1}, 2p^j\}$. 
    \end{itemize} 
\end{proposition}
\begin{proof}
    Let us fix some notation. For a subgroup $R$ of $\GL_2(\Z/p^{k+1}\Z)$, denote by $R'$ the image of $R$ under the projection $\pi:\GL_2(\Z/p^{k+1}\Z)\to \GL_2(\Z/p^k\Z)$. 

    Note that the number of $p^{k+1}$-isogenies can always be $0$, if we take $R=\pi^{-1}(R')$, so that $R$ contains all matrices congruent to the identity modulo $p^k$. Now suppose that the number of $p^{k+1}$-isogenies is at least $1$, and let $R \leq B_0(p^{k+1})$. Again, let $\alpha=\min_R(v_p(c-a))$ and $\beta=\min_R(v_p(b)).$ 

    If the number of $p^k$-isogenies is $p^k+p^{k-1}$, then by the proof of \Cref{neparno}, it follows that $\alpha$ and $\beta$ are at least $k$ and every matrix in $R'$ is scalar. If $\alpha=k+1$, and $\beta=k$, there are only improper solutions of type B equations, and hence $p^k$ solutions in total. If $\alpha=\beta=k+1$, there are $p^{k+1}+p^k$ solutions. If $\alpha=k$ and $\beta=k+1$, there are $p^k$ solutions of type A and hence $2p^k$ solutions in total. If $\alpha=\beta=k$, there are either $p^k$ solutions of type A and $p^k$ solutions of type B, or there are $p^k$ improper solutions of type B. 
    
    To prove that each possibility arises, take $R$ to be trivial to obtain $p^{k+1}+p^k$, generated by $\sm{1 & p^k \\ 0 & 1}$ to obtain $p^k$, generated by $\sm{1 & 0 \\ 0 & 1+p^k}$ to obtain $2p^k$.

    Suppose the number of $p^k$-isogenies is $2p^j$ for $0\leq j <k$. Then there are either solutions of type A or proper solutions of type B. In any case, one has $\alpha=j$, $\beta\leq \alpha$ and it follows from the proof of \Cref{neparno} that the number of $p^{k+1}$-isogenies is either $p^j$ or $2p^j$. The number $2p^j$ is achieved by taking $R$ generated by $\sm{1 & 0 \\ 0 & 1+p^j}$. The number $p^j$ is achieved by taking $R$ generated by $\sm{1 & p^k \\ 0 & 1+p^j}$ and $\sm{1 & 0 \\ 0 & 1+p^j}$.

    Finally, suppose that the number of $p^k$-isogenies is $p^j$ for $j<k$. Then there are no type A solutions mod $p^k$, so there are also no type A solutions mod $p^{k+1}$. Furthermore, all solutions mod $p^k$ satisfy $(c-a)y=by^2=0 \pmod {p^k}$. 
    
    Suppose that all solutions mod $p^{k+1}$ are improper. Then, there are $p^{\min(\alpha, \lfloor (k+1+\beta)/2 \rfloor)}$ solutions. If $j\geq \frac{k-1}{2}$, taking $\alpha=k+1$, and $\beta=2j+1-k$ so that $\lfloor (k+\beta)/2\rfloor=j$ and $\lfloor (k+\beta+1)/2\rfloor=j+1$ gives $p^{j+1}$ solutions. This can be achieved if $R$ is generated by $\sm{1 & p^{2j+1-k} \\ 0 & 1}.$  
    
    If  $\alpha\leq \frac{k+\beta}{2}$, then $j=\alpha$, and we get $p^j$ solutions.  This can be achieved if $R$ is generated by $\sm{1 & p^{k-1} \\
    0 & p^j+1}$ and $\sm{1 & 0 \\ 0 & p^j+1}$.

    Now suppose that there exist proper solutions mod $p^{k+1}$. Then $\alpha<\frac{\beta+k+1}{2}$, and hence $\alpha=j$. Any proper solution satisfies $v_p(y)=\alpha-\beta$. Since there are no proper solutions mod $p^k$, we must have $(c-a)y=by^2\equiv 0 \pmod{p^k}$ for any solution, and taking $v_p$ yields $2\alpha-\beta=k$, so $j=\alpha=\frac{k+\beta}{2}\geq \frac{k}{2}$. This can be achieved by taking $R$ generated by $\sm{1 & p^{2j-k} \\ 0 & 1+p^j}$. In this case, we get $2p^j$ solutions.
\end{proof}


\begin{remark}
    Note that one can also think about the problem of $p^k$-isogenies geometrically. There is a bijection between cyclic subgroups of $(\Z/p^k\Z)^2$ of order $p^k$ and points on $\PP^1(\Z/p^k\Z)$. The group $\PP \GL_2(\Z/p^k\Z)$ acts on $\PP^1(\Z/p^k\Z)$ by $$\sm{a & b \\c & d}(x:y)=(ax+by: cx+dy).$$ 
    For $k=1$, these are projective linear transformations on the projective line $\PP^1(\F_p)$, and any such transformation that fixes three points fixes all points, which is the reason for there only being $0,1,2$ or $p+1$ possible $p$-isogenies. As we saw in our results, this is not true for $k>1$. However, there is a sense in which this generalizes, as shown by the next proposition.
\end{remark}

\begin{proposition}\label{ako 3 onda sve}
Let $p$ be an odd prime. Suppose that a matrix $\sm{a & b \\ c & d}\in \PP \GL_2(\Z/p^k\Z)$ fixes three points on $\PP^1(\Z/p^k\Z)$ whose reductions modulo $p$ on $\PP^1(\F_p)$ are distinct. Then $\sm{a & b \\ c & d}$ fixes all points on $\PP^1(\Z/p^k\Z)$.    
\end{proposition}
\begin{proof}
    Note that the reduction mod $p$ map $\PP^1(\Z/p^k\Z)\to \PP^1(\F_p)$ is well defined: any point $(x:y)\in \PP^1(\Z/p^k\Z)$ can be written either as $(x':1)$ or $(1:y')$. In that case, the reduction mod $p$ is defined as $(x' \pmod p:1)$ or $(1:y' \pmod p)$.

    Since $\sm{a & b \\ c & d}$ fixes at least one point, we can conjugate it to a matrix of the form $\sm{a & b \\ 0 & c}$, so that it fixes the point $(1:0)$. Then, by the condition, it also fixes two points of the form $(x:1)$ and $(y:1)$ with $x\not \equiv y \pmod p$. We then have $(ax+b:c)=(x:1)$ and $(ay+b:c)=(y:1)$, so $(c-a)y=(c-a)x=b$, from where it follows that $a=c$ and $b=0$, and hence the matrix is scalar, so it fixes all points.
\end{proof}
Rephrasing in terms of isogenies, we obtain the following.
\begin{corollary}
    Let $p$ be an odd prime and $k$ a positive integer. Let $E$ be an elliptic curve over a number field $K$, and suppose that $E$ has three cyclic $p^k$-isogenies over $K$, with kernels $C_1, C_2, C_3$. If $C_i \cap C_j$ is trivial for $1\leq i<j\leq 3$, then $E$ has all $p^k+p^{k-1}$ isogenies defined over $K$.
\end{corollary}
\begin{proof}
    Consider the correspondence between cyclic subgroups of $E[p^k]$ of order $p^k$ and $\PP^1(\Z/ p^k\Z)$. Let $P_i$ be the point on $\PP^1(\Z/p^k\Z)$ corresponding to $C_i$. The condition $C_i \cap C_j=\{\mathcal O\}$ is equivalent to $P_i$ and $P_j$ being distinct modulo $p$. The claim now follows from \Cref{ako 3 onda sve}.
\end{proof}

\begin{remark}
    All results from this section are in fact true when we replace number fields with arbitrary fields of characteristic $0$. The proofs do not change at all. Over characteristic $p$, the situation is different. A nice expository article describing the structure of isogeny graphs over finite fields is \cite{volcanoes}.
\end{remark}

\section{Possible number of points above a $j$-invariant on some classes of modular curves}\label{sec4}

Having determined the possible values for the number of $p^k$-isogenies of an elliptic curve over a number field, we aim to do analogous work for the number of points on some other families of modular curves. We will discuss modular curves corresponding to Cartan subgroups and their normalisers.

\begin{definition} Let $p$ be an odd prime and $k$ a positive integer. \\ The \emph{split Cartan subgroup} of $\mathrm{GL}_2(\mathbb{Z}/p^k\mathbb{Z})$ is the subgroup of all diagonal matrices, 
$$
C_{\mathrm{s}}(p^k)=\left\{\begin{pmatrix}a&0\\[4pt]0&b\end{pmatrix}:a,b\in (\Z/p^k\Z)^\times\right\}.
$$
Let $\epsilon \in (\Z/p^k\Z)^\times$ be a non-square element. The \emph{nonsplit Cartan subgroup} of $\GL_2(\mathbb{Z}/p^k\mathbb{Z})$ is
$$
C_{\mathrm{ns}}(p^k)
  = \left\{
     \begin{pmatrix}
       a & \epsilon b \\
       b & a
     \end{pmatrix}
     : a,b \in \mathbb{Z}/p^k\mathbb{Z},\ 
       a^2 - \epsilon b^2 \in (\mathbb{Z}/p^k\mathbb{Z})^\times
    \right\}.
$$
\end{definition}
\begin{remark}
    Any conjugate of the split/nonsplit Cartan subgroups defined above is also referred to as split/nonsplit Cartan. The different choices of $\epsilon$ in the definition of $C_{\mathrm{ns}}$ yield conjugate groups. Since our results do not depend on conjugation, we will simply work with the above realisations of Cartan subgroups.
\end{remark}
The description of normalisers of the Cartan subgroups is well-known. 
\begin{lemma} \phantom{a}
    \begin{enumerate}[label=\textnormal{(\alph*)}]
        \item The normaliser $N_{\mathrm s}(p^k)$ of $C_{\mathrm s}(p^k)$ in $\GL_2(\Z/p^k\Z)$ is generated by $\sm{0 & 1 \\ 1 & 0}$ and $C_{\mathrm s}(p^k)$. The group $C_{\mathrm s}(p^k)$ has index $2$ inside $N_{\mathrm s}(p^k)$.
        \item The normaliser $N_{\mathrm{ns}}(p^k)$ of $C_{\mathrm{ns}}(p^k)$ in $\GL_2(\Z/p^k\Z)$ is generated by $\sm{1 & 0 \\ 0 & -1}$ and $C_{\mathrm{ns}}(p^k)$. The group $C_{\mathrm{ns}}(p^k)$ has index $2$ inside $N_{\mathrm{ns}}(p^k)$.
    \end{enumerate}
\end{lemma}
\begin{proof}
    Part (a) follows from direct calculation. For part (b), see \cite[Theorem 1.1]{alvarocm}.
\end{proof}

Recall that for a group $H\leq \GL_2(\Z/N\Z)$, we denote by  $S(H)$ the set of possible values that $$\frac{\#\{g \in \GL_2(\Z/N\Z) \mid gRg^{-1}\subseteq H\}}{\#H}$$ obtains as $R$ ranges over all subgroups of $\GL_2(\Z/N\Z)$. 

\begin{proposition}
    Let $p$ be an odd prime and $k$ a positive integer. Then $$S(C_{\mathrm s}(p^k))=\{ 0, p^{2k}+p^{2k-1}\}\cup \{2p^{2j} \mid j\in \{0,1,\ldots, k-1\}\}.
    $$
\end{proposition}
\begin{proof}
    Let $H=C_{\mathrm s}(p^k)$. Let $R$ be a subgroup of $H$. Let $g=\sm{u & v \\ w & z}$. Then $gRg^{-1}\subseteq H$ if and only if $$\sm{u & v \\ w & z} \sm{a & 0 \\ 0 & b} \sm{z & -v \\ -w & u}$$ is a diagonal matrix for every $\sm{a & 0 \\ 0 & b} \in R$. Calculating the top right and bottom left entry in the product and equating them with $0$ gives us the following two equations: \begin{align*}
        uv(a-b)&=0, \\
        wz(a-b)&=0.
    \end{align*}
    Let $\alpha$ denote $\min_{g \in R} v_p(a-b)$. If $\alpha=k$, $H$ consists only of scalar matrices and is hence normal in $\GL_2(\Z/p^k\Z)$, so $gRg^{-1}$ is always contained in $R$. Suppose $\alpha<k$. Then the equations imply that $gRg^{-1}$ is contained in $H$ if and only if $uv$ and $wz$ are divisible by $p^{k-\alpha}$. Since $(u, v)$ is a row of an invertible matrix, at most one among $u$ and $v$ is divisible by $p$, and analogously for $(w, z)$. 
    
    If $u$ is divisible by $p^{k-\alpha}$, then $w$ and $v$ are invertible modulo $p$ and $z$ is divisible by $p^{k-\alpha}$. There are $p^{\alpha}$ choices for $u$, $p^\alpha$ choices for $z$, $\varphi(p^k)$ choices for $v$ and $\varphi(p^k)$ choices for $w$. We get the same number in total if $v$ is divisible by $p^{k-\alpha}$. 

    In total, the number of elements $g$ such that $gRg^{-1}\subseteq H$ is $2p^{2\alpha} \cdot \varphi(p^k)^2$. Dividing by the size of $H$, which is $\varphi(p^k)^2$, we obtain $2p^{2\alpha}$.

    To complete the proof, for each $j \in \{0,1,\ldots, k-1\}$ we must construct a group $R$ such that $\min_{g \in R} v_p(a-b)=j$. To achieve this, one can take $R$ to be the cyclic group generated by $\sm{p^j+1 & 0 \\
     0 & 1}$.
\end{proof}

\begin{proposition}
    Let $p$ be an odd prime and $k$ a positive integer. Let 
$$
\begin{aligned}
A_k(p)&=\{0\}\cup \{p^{2j} \mid 0\le j<k\}\cup\left\{\frac{p^{2k}+p^{2k-1}}{2}\right\},\\
B_k(p)&=\{p^{j} \mid 0\le j<k\},\\
C_k(p)&=\left\{\frac{p^{k}+p^{k-1}}{2},\ \frac{p^{k}-p^{k-1}}{2}+1\right\},\\
D_k(p)&=\begin{cases}
\{2\}, & \text{if } p\equiv 3\pmod{4},\\[4pt]
\{3\}, & \text{if } p\equiv 1\pmod{4}.
\end{cases}
\end{aligned}
$$
  Then $$S(N_{\mathrm s}(p^k))=A_k(p)\cup B_k(p)\cup C_k(p) \cup D_k(p).$$
\end{proposition}
\begin{proof}
    Let $H=N_{\mathrm s}(p^k)$. Note that any matrix in $H$ is either of the form $\sm{a & 0 \\ 0 & b}$ or $\sm{0 & c \\ d & 0}$.  Let $R$ be a subgroup of $H$.
    Let $g=\sm{u & v \\ w & z}$. We now calculate the matrices $g \sm{a & 0 \\0 & b}g^{-1}$ and $g \sm{0 & c \\d & 0}g^{-1}$. We have $$
g
\begin{pmatrix} a & 0 \\[4pt] 0 & b \end{pmatrix}
g^{-1}
  = \frac{1}{uz - vw}
    \begin{pmatrix}
      u z a - v w b & u v (b - a) \\[4pt]
      w z (a - b) & u z b - v w a
    \end{pmatrix},
$$
and similarly
$$
g
\begin{pmatrix} 0 & c \\[4pt] d & 0 \end{pmatrix}
g^{-1}
  = \frac{1}{uz - vw}
    \begin{pmatrix}
      v z d - u w c & u^2 c - v^2 d \\[4pt]
      z^2 d - w^2 c & u w c - v z d
    \end{pmatrix}.
$$
Consider the following four systems of equations:

\textbf{System A1:}
\[
\begin{cases}
auz = bvw, \\[4pt]
buz = avw.
\end{cases}
\]

\textbf{System A2:}
\[
\begin{cases}
wz(b-a)=0, \\[4pt]
uv(b-a)=0.
\end{cases}
\]

\textbf{System B1:}

\[
\begin{cases}
vzd=uwc.
\end{cases}
\]

\textbf{System B2:}

\[
\begin{cases}
u^2c=v^2d, \\[4pt]
w^2c =z^2d.
\end{cases}
\]

A matrix $g=\sm{u & v \\ w & z}\in \GL_2(\Z/p^k\Z)$ satisfies $gRg^{-1}\subseteq H$ if and only if for each $\sm{a & 0 \\ 0 & b}\in R$, $g$ is a solution of system A1 or system A2, and for each $\sm{0 & c \\ d & 0} \in R$, $g$ is a solution of system B1 or system B2. We will call any such matrix $g$ a solution. We say a solution is trivial if $g \in H$. Note that $g$ cannot simultaneously be a solution of systems A1 and A2, or B1 and B2. However, for some $\sm{a & 0 \\ 0 & b}$ it may be a solution of A1, while for some other $\sm{a & 0 \\ 0 & b}$ it is a solution of A2 etc.

\noindent \textbf{Case I.} There is a solution $g$ which is a solution of some system A1. Adding up the two equalities, we obtain $(a+b)(uz-vw)=0$, so $a+b=0$ since $uz-vw=\det(g)$ is invertible. It follows that $b=-a$. Furthermore, we obtain $uz+vw=0$, from where it follows that $u,z,v,w$ are all invertible. Thus, if $g$ is a solution of system A2 for some other $\sm{a & 0 \\ 0 & b}$, it follows that $a=b$.

In conclusion, if there is a solution which is a solution of system A1, then $a/b\in\{-1,1\}$ for all $\sm{a & 0 \\ 0 & b} \in R$, and $-1$ is achieved at least once. Also, taking $\sm{a & 0 \\ 0 & b}$ with $b=-a$, it follows that \emph{any} solution $\sm{u & v \\ w & z}$ for which some of $u,v,w,z$ is not invertible is an element of $H$, i.e.\ a trivial solution.

\noindent \textbf{Case I.a} There are no matrices of the form $\sm{0 & c \\d & 0}$ in $R$. We need to count the number of solutions to $uz+vw=0$ among invertible matrices $\sm{u & v \\ w & z}.$ We can pick $u$, $z$ and $w$ to be any invertible elements of $\Z/p^k\Z$, and $v$ is then uniquely determined. This gives us $\varphi(p^k)^3$ solutions in total. The solutions which are not solutions of system A1 are always solutions of system A2, which turns into $wz=uv=0$ for $b=-a$. The solutions to this system are exactly the trivial ones, i.e.\ the elements of $H$, so we get $2\varphi(p^k)^2$ solutions. Dividing by $\#H=2\varphi(p^k)^2$ gives us the final number $\frac{1}{2}\varphi(p^k)+1=\frac{p^k-p^{k-1}}{2}+1$.

\noindent \textbf{Case I.b} There are matrices of the form $\sm{0 & c \\d & 0}$ in $R$.
Note that the systems B1 and B2 only depend on the value of $\frac{d}{c}$. Furthermore, note that if $t$ is achieved as $\frac{d}{c}$ for some $\sm{0 & c \\ d & 0} \in R$, then $-t$ is achieved for $\sm{0 & c \\ d & 0}\sm{a & 0 \\0 & -a}$. Thus, the set of values of $\frac{d}{c}$ is closed under multiplication by $-1$. We now claim that there are exactly two values of $\frac{d}{c}$ which are achieved. Namely, since $g$ is a solution with $u,v,w,z$ invertible, one has $\frac{d}{c} \in \{u^2/v^2, uw/vz\}$, depending on whether $g$ is a solution of B1 or B2. Furthermore, if $g$ is a solution of system B1 for some $\sm{0 & c \\d & 0}$, it is a solution of system B2 for $\sm{0 & c \\ d & 0}\cdot \sm{a & 0 \\0 & -a}$. Let $\{t, -t\}$ be the set of those values. Note that at least one among $t$ or $-t$ is a square (this follows from system B2 having a solution). Without loss of generality, we may assume that $t$ is a square.

\noindent \textbf{Case I.b.(i)} $-1$ is not a square.
Let $\sm{u & v \\w & z}$ be any nontrivial solution. Then it is a solution of some system A1 so $uz+vw=0$. Furthermore, for $\sm{c & 0 \\ 0 & d}$ such that $\frac{d}{c}=t$ it is a solution of system B2 and for $\sm{c & 0 \\ 0 & d}$ such that $\frac{d}{c}=-t$ it is a solution of system B1. To summarize, we have \[
\begin{cases}
u^2=tv^2, \\[4pt]
w^2 =tz^2, \\[4pt]
-vzt=uw, \\[4pt]
uz=-vw.
\end{cases}
\]
We have $u\in\{\sqrt{t}v, -\sqrt{t}v\}$ and $w\in \{\sqrt{t}z, -\sqrt t z\}$. To satisfy the third equation, the sign choices need to be distinct. In this way, the fourth equation is automaticaly true. Thus, we have $\varphi(p^k)^2$ choices for $v$ and $z$, and two choices for a combination of $u$ and $w$ with distinct signs. In total, there are $2\varphi(p^k)^2$ nontrivial solutions. Dividing by $\#H$ and adding the trivial solution, we obtain the number $2$.

\noindent \textbf{Case I.b.(ii)} $-1$ is a square. In this case, there is an extra choice, as $t$ and $-t$ may change their roles for some $g$. In total, we get $4\varphi(p^k)^2$ solutions. Dividing by $\#H$ and adding the trivial solution, we obtain the number $3$.

\noindent \textbf{Case II.} All solutions $g$ are solutions of systems A2. As before, let $\alpha=\min_{g \in R} v_p(b-a)$. 

\noindent \textbf{Case II.a.} $\alpha=k$. This means that all matrices of the form $\sm{a & 0 \\ 0 & b}\in R$ satisfy $a=b$. In this case, any matrix $g$ is a solution to the system A2. It remains to count the number of $g$ which are solutions to either B1 or B2. Note that if there are no matrices of the form $\sm{0 & c  \\ d & 0}$, all $g$ are solutions and we get the number $\frac{p^{2k}+p^{2k-1}}{2}$. Now suppose that we are not in this case.

Note that if $\sm{0 & c_1 \\ d_1 & 0}$ and $\sm{0 & c_2 \\ d_2 & 0}$ are in $H$, their product $\sm{c_1d_2 & 0 \\ 0 & d_1c_2}$ is also in $H$, so $c_1d_2=c_2d_1$. This means that the ratio $\frac{d}{c}$ is equal to some constant $t$, and the systems B1 and B2 do not depend on the choice of a matrix in $H$. Thus, it just remains to count the number of solutions to B1 and B2 separately and add the numbers up.

\noindent \textbf{Case II.a.(i)} $t$ is a square.
We first count the number of solutions of system B1. If $u$ is not invertible, then $v$ and $w$ are, so $z=\frac{uw}{tv}$ is uniquely determined. In total, we have $p^{k-1}$ choices for $u$ and $\varphi(p^k)^2$ choices for $v$ and $w$, which gives us $p^{k-1}\varphi(p^k)^2$.

If $u$ is invertible, then $w=\frac{vzt}{u}$ is uniquely determined. Furthermore, we must ensure that the determinant is invertible, which, after some calculation, simplifies to $u^2-tv^2$ and $z$ being invertible. There are $\varphi(p^k)$ choices for $z$. Since $t$ is a square, there are $2p^{k-1}$ values of $v$ for which $u^2-tv^2$ is divisible by $p$, and hence $p^k-2p^{k-1}$ values for which it is invertible. We thus have $(p^k-2p^{k-1})\varphi(p^k)^2$ choices when $u$ is invertible. In total, system B1 has $\varphi(p^k)^2(p^k-p^{k-1})$ solutions.

The system B2 has $2\varphi(p^k)^2$  solutions: we pick $u$ and $w$ freely in $\varphi(p^k)^2$ ways, and then $v^2$ and $z^2$ are uniquely determined, which means $v$ and $z$ are determined up to sign. As in the previous cases, we need to take one of the $2$ different choices of signs to ensure that the determinant $uz-vw$ is invertible. 

Summing those values and dividing by $2\varphi(p^k)^2$, we obtain $\frac{p^k-p^{k-1}}{2}+1$.

\noindent \textbf{Case II.a.(ii)} $t$ is not a square. In this case, system B2 has no solutions. For system B1, the analysis is analogous to the previous case, with the only difference that when $u$ is invertible, $u^2-tv^2$ is invertible for all values of $v$. In total, there are $p^{k-1}\varphi(p^k)^2+p^k\varphi(p^k)^2$ solutions, and we obtain the number $\frac{p^k+p^{k-1}}{2}$.

\noindent \textbf{Case II.b.} $\alpha<k$. 
Note that then either $u$ and $z$ must both be divisible by $p$, or $v$ and $w$ are both divisible by $p$. However, this means that system B2 never has solutions, as each of $u,v,w,z$ would be divisible by $p$. Thus, every matrix is a solution of system B1.

\noindent \textbf{Case II.b.(i)} There are no matrices of the form $\sm{0 & c \\ d & 0}$ in $R$. In this case, the problem is analogous to the one for $C_{\mathrm s}(p^k)$, and we get the number $p^{2\alpha}$.

\noindent \textbf{Case II.b.(ii)} There are matrices of the form $\sm{0 & c \\d & 0}$ in $R$. 
Note that either $u$ and $z$ or $v$ and $w$ are divisible by $p^{k-\alpha}$. Without loss of generality, assume that $u$ and $z$ are divisible by $p^{k-\alpha}$, and then $v$ and $w$ are invertible.

Consider the system B1 for a fixed $\frac{d}{c}$. We have $$vzd=uwc.$$ Take any $u$ divisible by $p^{k-\alpha}$, and any two invertible $v$ and $w$. Then $z$ is uniquely determined. We get $p^\alpha\cdot \varphi(p^k)^2$ solutions. Similarly, there is the same number of solutions when $v$ and $w$ are divisible by $p^{k-\alpha}$ and $u$ and $z$ are invertible. In total, dividing by $2\varphi(p^k)^2$, we get $p^\alpha$. This is all for a fixed value of $\frac{d}{c}$. However, note that if $\sm{0 & c_1 \\ d_1 & 0}$ and $\sm{0 & c_2 \\ d_2 & 0}$ are in $R$, their product $\sm{c_1d_2 & 0 \\0 & d_1c_2}$ satisfies $p^\alpha \mid c_1d_2-d_1c_2$, and it follows that $\frac{d_1}{c_1}\equiv \frac{d_2}{c_2} \pmod{p^\alpha}$, meaning that every choice of $\frac{d}{c}$ gives the same system B1, and the solutions of the systems are the same.

We have now gone through all cases. We now need to check that each of the cases does occur for some subgroup.

\noindent \textbf{Case I.a} Take $R$ to be generated by $\sm{1 & 0 \\ 0 & -1}.$

\noindent \textbf{Case I.b} Take $R$ to be generated by $\sm{1 & 0 \\ 0 & -1}$ and $\sm{0 & t \\ 1 & 0}$ for some square $t$. Depending on whether $-1$ is a square, we end up in Case I.b.(i) or I.b.(ii).

\noindent \textbf{Case II.a} Take $R$ to be generated by $\sm{0 & t \\ 1 & 0}$. Depending on whether $t$ is a square or not, we end up in Case II.a.(i) or Case II.a.(ii).

\noindent \textbf{Case II.b.(i)} This case was already discussed for $C_{\mathrm s}(p^k)$.

\noindent \textbf{Case II.b.(ii)} Take $R$ to be generated by $\sm{1  & 0 \\ 0 & p^j+1}$ and $\sm{0 & 1 \\ 1 & 0}$.
\end{proof}

\begin{proposition}
    Let $p$ be an odd prime and $k$ a positive integer. Then $$S(C_{\mathrm{ns}}(p^k))=\{0\} \cup \{2p^{2j} \mid 0\leq j<k\} \cup \{p^{2k}-p^{2k-1} \}.$$
\end{proposition}
\begin{proof}
    Let $H=C_{\mathrm{ns}}(p^k)$. We start by calculating a suitable set of coset representatives for $\GL_2(\Z/p^k\Z)/H$. 
     \begin{lemma}
         The matrices of the form $\sm{1 & v \\ 0 & z}$ for $v\in \Z/p^k\Z$ and $u\in (\Z/p^k\Z)^\times$ form a full list of coset representatives of $C_{\mathrm{ns}}(p^k)$ inside $\GL_2(\Z/p^k\Z)$. 
     \end{lemma}
     \begin{proof}
         The number of listed matrices coincides with the index $[\GL_2(\Z/p^k\Z):C_{\mathrm{ns}}(p^k)]$, so it suffices to prove that each matrix can be written in the form $h\cdot r$, where $r$ is a matrix from the list and $h\in C_{\mathrm{ns}}(p^k)$. 

         Start with any matrix $\sm{u & v \\ w & z} \in \GL_2(\Z/p^k\Z)$. If $u\equiv 0 \pmod p$, then $w$ is invertible, so if we multiply from the left by $\sm{0 & \epsilon \\ 1 & 0}$, we get a matrix with top left entry $\epsilon w$, which is invertible. Thus, we may assume $u$ is invertible.

         Furthermore, take $a$ and $b$ such that $a$ is invertible and $bu+aw=0$. Then multiplying on the left by $\sm{a & \epsilon b \\ b & a}$ gives us a matrix whose bottom left entry equals $0$. Finally, multiplying by a suitable scalar matrix yields a matrix whose top left entry is $1$ and bottom left entry is $0$, which proves the claim. 
     \end{proof}
    Thus, for a subgroup $R$ of $H$, we need to find the number of matrices of the form $g=\sm{1 & v \\ 0 & z}$ such that $gRg^{-1}\subseteq H$. 
     
     We start by calculating $\sm{1 & v \\0 & z}\sm{a & \epsilon b \\ b & a}\sm{1 & v \\ 0 & z}^{-1}$. We obtain, up to multiplication by scalars, $$\begin{pmatrix} az+bvz & b(\epsilon - v^2) \\
bz^2 & az-bvz\end{pmatrix}.$$
This matrix lies in $H$ if and only if the following holds: \begin{align*}
bv &=0, \\
b(z^2-1)&=0.
\end{align*}

We see that the number of solutions $\sm{1 & v \\0 & z}$ depends only on the least possible value of $v_p(b)$. Let $\beta=\min_{g \in R}(v_p(b))$. Note that taking $R$ to be generated by $\sm{1 & \epsilon p^j \\ p^j & 1}$ achieves $\beta=j$, so $\beta$ can take on any value between $0$ and $k$. 

If $\beta=k$, every matrix $\sm{1 & v \\0 & z}$ is a solution, and we obtain the number $p^{2k}-p^{2k-1}$. Now suppose $\beta<k$. Then $v$ is any number divisible by $p^{k-\beta}$, so there are $p^\beta$ options. On the other hand, either $z-1$ or $z+1$ is divisible by $p^{k-\beta}$, so there are $2p^{\beta}$ options. In total, there are $2p^{2\beta}$ solutions.
\end{proof}

\begin{proposition}
     Let $p$ be an odd prime and $k$ a positive integer.  Let $$
\begin{aligned}
A'_k(p)&=\{0\} \cup \left\{ \frac{p^{2k}-p^{2k-1}}{2} \right\}\cup \{p^{2j} \mid j\in \{0,1,\ldots,k-1\}\},\\
B'_k(p)&=\left\{\frac{p^k+p^{k-1}}{2}+1\right\},\\
C'_k(p)&=\begin{cases}
\{2\}, & \text{if } p\equiv 1\pmod{4},\\[4pt]
\{3\}, & \text{if } p\equiv 3\pmod{4},
\end{cases}\\
D'_k(p)&=\left\{1,\ \frac{p^k+p^{k-1}}{2},\ \frac{p^k-p^{k-1}}{2}\right\},\\
E'_k(p)&=\{p^j \mid 0<j<k\},\\
F'_k(p)&=\begin{cases}
\{p^{\lfloor k/2 \rfloor}+1\}\ \cup\ \{\,2p^t+1 \mid 0<2t<k\,\}, & \text{if } p\equiv \pm 3 \pmod 8,\\[4pt]
\emptyset, & \text{if } p\equiv \pm 1 \pmod 8,
\end{cases}\\
G'_k(p)&=\begin{cases}
\{1,2,3\}, & \text{if } p>3,\\[4pt]
\{1,3\}, & \text{if } p=3,
\end{cases}\\
H'_k(p)&=\begin{cases}
\{1,2\}, & \text{if } p\equiv 1\pmod{4},\\[4pt]
\{1\}, & \text{if } p\equiv 3\pmod{4}.
\end{cases}
\end{aligned}
$$ Then 
    $$S(N_{\mathrm{ns}}(p^k))=\bigcup_{L=A}^H L'_k(p).$$
\end{proposition}

\begin{proof}
    Let $H=N_{\mathrm{ns}}(p^k)$. Let $R$ be a subgroup of $H$. The list of coset representatives for $N_{\mathrm{ns}}(p^k)$ is the same as the one for $C_{\mathrm{ns}}(p^k)$, except $\sm{1 & v \\ 0 & z}$ and $\sm{1 & v \\ 0 & -z}$ yield the same coset. We will find the number of matrices $g=\sm{1 & v \\ 0 & z}$ satisfying $gRg^{-1}\subseteq H$, and then divide this number by $2$.
    
   We have already calculated the form of conjugates in $C_{\mathrm{ns}}(p^k)$ by $\sm{1 & v \\ 0 & z}$ as 
    $$\begin{pmatrix} az+bvz & b(\epsilon - v^2) \\
bz^2 & az-bvz\end{pmatrix}.$$
Similarly, calculating $\sm{1 & v \\0 & z}\sm{a & \epsilon b \\ -b  & -a}\sm{1 & v \\ 0 & z}^{-1}$ yields 
  $$\begin{pmatrix} az-bvz & bv^2-2av+\epsilon b \\
-bz^2 & bvz-az\end{pmatrix}.$$

\noindent \textbf{Case I.} $R$ is contained in $C_{\mathrm{ns}}(p^k)$. 

\noindent \textbf{Case I.a.} For every solution $g$, one has $gRg^{-1}\subseteq C_{\mathrm{ns}}(p^k)$. This case is equivalent to the case from part (c), and we obtain the same set of numbers, except we divide them by $2$ because the index is twice smaller.

\noindent \textbf{Case I.b.} There exists a solution $g$ with $gRg^{-1}\not \subseteq C_{\mathrm{ns}}(p^k)$. This means that $R$ contains a matrix $m=\sm{a & b \\ \epsilon b & a}$ such that  $$\begin{pmatrix} az+bvz & b(\epsilon - v^2) \\
bz^2 & az-bvz\end{pmatrix}.$$ belongs to $H\setminus C_{\mathrm{ns}}(p^k)$. This is equivalent to 
\[
\begin{cases}
2az=0, \\[4pt]
b(\epsilon z^2+\epsilon)=bv^2.
\end{cases}
\]
Since $2z$ is invertible, it follows that $a=0$. Hence, $b$ is invertible and we get the equation \begin{equation} \label{bla1}\epsilon z^2+\epsilon=v^2.\end{equation}

Furthermore, if $g$ is a solution such that $gmg^{-1}$ lies in $C_{\mathrm{ns}}(p^k)$, it follows from the proof of part (c) that $g\in N_{\mathrm{ns}}(p^k)$ (note that $\beta=0$ in that case). Thus, there is exactly one class of solutions $g$ such that $gmg^{-1} \in C_{\mathrm{ns}}(p^k)$. Any other solution satisfies \eqref{bla1}.

Now suppose that $R$ contains a nonscalar matrix $\sm{a & \epsilon b \\b & a}$ with $a\neq 0$. This means that the conjugate of this matrix by $\sm{1 & v \\ 0 & z}$ lies in $C_{\mathrm{ns}}(p^k)$. Part (c) then implies that $v\equiv 0 \pmod p$ and $z^2\equiv 1 \pmod p$. However, then \eqref{bla1} implies $2\epsilon \equiv 0 \pmod p$, a contradiction. Thus, it suffices to count the number of solutions of the equation \eqref{bla1} for which $z$ is invertible. However, note that $z$ is always invertible, since otherwise we would have $\epsilon \equiv v^2 \pmod p$.

It follows from a standard character calculation argument that the number of solutions $(v,z)$ modulo $p$ to  \eqref{bla1} is $p-\chi_p(\epsilon)=p+1$, where $\chi_p$ is the quadratic character mod $p$. To go from $p$ to $p^k$, note that $z$ is always invertible, so Hensel's lemma implies that any solution in $\Z/p\Z$ has exactly $p^{k-1}$ lifts to $\Z/p^k\Z$. In total, after dividing by $2$ to account for the smaller index, we obtain $\frac{p^k+p^{k-1}}{2}$ solutions. However, we must also add the number of matrices $m=\sm{1 & v \\ 0 & z}$ such that $m\sm{0 & \epsilon b \\\ -b & 0}m^{-1}$ lies in $C_{\mathrm{ns}}(p^k)$. This is equivalent to $v=0$ and $z^2=1$, and since $z$ and $-z$ define the same coset, the total number of solutions is $\frac{p^k+p^{k-1}}{2}+1$.

This number is achieved by taking $R$ to be the group generated by $\sm{0 & \epsilon  \\ 1 & 0}.$

\noindent \textbf{Case II.} $R$ is not contained in $C_{\mathrm{ns}}(p^k)$. Note that $R_1:=R\cap C_{\mathrm{ns}}(p^k)$ is an index $2$ subgroup of $R$. We know that either any solution $g$ satisfies $gR_1g^{-1}\subseteq C_{\mathrm{ns}}(p^k)$, or any matrix $\sm{a & \epsilon b \\b & a}$ in $R_1$ satisfies either $a=0$ or $b=0$. 

Let $R_2=R\setminus R_1$. Then $R_2=rR_1$ for some $r=\sm{a_0 & \epsilon b_0 \\ -b_0 & -a_0}\in R_2$.  

\noindent \textbf{Case II.a.} There exists a solution $g$ such that $grg^{-1}\in C_{\mathrm{ns}}(p^k)$. We claim that $R_1$ consists only of matrices $\sm{a & \epsilon b \\ b & a}$ such that $a=0$ or $b=0$.

 Since $grg^{-1} \in C_{\mathrm{ns}}(p^k)$, we have $$\begin{cases}
    a_0z=b_0vz, \\
    -\epsilon b_0z^2=b_0v^2-2a_0v+\epsilon b_0.
\end{cases}    
$$
Since $z$ is invertible, this simplifies to $$\begin{cases}
    a_0=b_0v, \\
    -\epsilon b_0z^2=-b_0v^2+\epsilon b_0.
\end{cases}$$
Note that $b_0$ is invertible, since otherwise $a_0$ would not be. Thus, $-\epsilon z^2+v^2=\epsilon.$
Consider a matrix $m=\sm{a & \epsilon b \\ b & a}$ in $R_1$. If $gmg^{-1}\not \in C_{\mathrm{ns}}(p^k)$, then $a=0$. Otherwise, if $b\neq 0$, we can use the same argument as in part (c), to obtain $v\equiv 0 \pmod p$ and $z^2\equiv 1 \pmod p$. However, this implies $-\epsilon \equiv \epsilon \pmod p$, a contradiction.   

Thus, each matrix $\sm{a &  \epsilon b \\ b & a}$ in $R_1$ satisfies either $a=0$ or $b=0$.

\noindent \textbf{Case II.a.(i)} Each matrix in $R_1$ is a scalar matrix. This means that a solution is any matrix $g$ such that $grg^{-1} \in C_{\mathrm{ns}}(p^k)$ or $grg^{-1} \in H\setminus C_{\mathrm{ns}}(p^k)$. Thus, we need to find the sum of the numbers of solutions to the system $$\begin{cases}
a_0=b_0v, \\
v^2=\epsilon z^2+\epsilon,
\end{cases}$$
and add to that number the number of solutions of the equation\begin{equation} \label{bla2}
    \epsilon b_0 z^2=b_0 v^2-2a_0v+\epsilon b_0,
\end{equation}
with $a_0$ and $b_0$ fixed and at least one of them invertible, and with the assumption that the first system has at least one solution. Note that in fact the first system has at most one solution, since $v$ is uniquely determined from the first equation, and $z^2$ is uniquely determined from the second equation. Also, this equation does have a solution if and only if $$\frac{a_0^2}{\epsilon b_0^2}-1$$ is a square modulo $p$. Note that this is always possible to achieve, as there always exist elements $x$ and $x+1$ in $(\Z/p\Z)^\times$ such that $x$ is a residue and $x+1$ a nonresidue.

It remains to find the number of solutions to \eqref{bla2} for which $z$ is invertible. We can divide the equation by $b_0$ since it is also invertible. We get $$\epsilon z^2=v^2-2\frac{a_0}{b_0}v+\epsilon.$$
Again, using standard character counting argument, it follows that there are $p+1$ solutions modulo $p$ to this equation. Note that there are no solutions with $z=0$, since we would have $$\left(v-\frac{a_0}{b_0} \right)^2=\epsilon\left(\frac{a_0^2}{\epsilon b_0^2}-1\right),$$ and the right hand side is not a square mod $p$. Thus, we get $p+1$ solutions modulo $p$ in this case. There are $p^{k-1}$ lifts of each solution, so we obtain $p^k+p^{k-1}$ solutions.  Dividing by $2$ to account for $z$ and $-z$ being the same, and adding $1$ for the solution of the system, we again obtain the number $\frac{p^k+p^{k-1}}{2}+1$.

\noindent \textbf{Case II.a.(ii)} There exists a nonscalar matrix in $R_1$.

As in the previous case, we conclude that there is exactly one solution to $$\begin{cases}
a_0=b_0v, \\
v^2=\epsilon z^2+\epsilon,
\end{cases}.$$

Consider a nonscalar matrix $\sm{0 & \epsilon b \\ b & 0} \in R_1$. Note that if $\sm{1 & v \\0 & z}$ is the solution to the system above, we have $\sm{1 & v \\0 & z }\sm{0 & \epsilon b \\ b & 0}\sm{1 & v \\ 0 & z}^{-1}\in H\setminus C_{\mathrm{ns}}(p^k)$, so the solution to the system remains valid. Also, there is always the trivial solution $v=0$, $z^2=1$. 

Suppose that there is some other nontrivial solution $(v_0, z_0)$. Then $\epsilon z_0^2+\epsilon= v_0^2$ by looking at the conjugate of $\sm{0 & \epsilon b \\ b & 0 }$. On the other hand, looking at the conjugate of $r=\sm{a_0 & \epsilon b_0 \\ -b_0 & -a_0}$, it follows that $$\epsilon b_0 z_0^2=b_0 v_0^2-2a_0v_0+\epsilon b_0.$$ We again divide by $b_0$, and then eliminate $\epsilon z^2$ from the equation to obtain $$v_0^2-\epsilon=v_0^2-2\frac{a_0}{b_0}v_0+\epsilon,$$ from where it follows that $$v_0=\frac{\epsilon b_0}{a_0}.$$
Plugging this back in the expression $\epsilon z_0^2+\epsilon=v_0^2$, it follows that $$\epsilon z_0^2+\epsilon=\frac{b_0^2\epsilon^2}{a_0^2},$$ and finally $$z_0^2=\frac{\epsilon b_0^2}{a_0^2}-1.$$
This equation has a solution if and only if $\frac{\epsilon b_0^2}{a_0^2}-1$ is a square, which is equivalent with $\epsilon b_0^2-a_0^2$ being a square. Since $\frac{a_0^2}{\epsilon b_0^2}-1$ is a square, it follows that $a_0^2-\epsilon b_0^2$ is not a square. Thus, it depends on whether $-1$ is a square or not. 

If $-1$ is a square, then $\epsilon b_0^2-a_0^2$ is not a square and we obtain $2$ solutions in total. 

If $-1$ is not a square, then $\epsilon b_0^2-a_0^2$ is a square and we obtain an additional solution, for a total of $3$ solutions.

\noindent \textbf{Case II.b.} All solutions $g$ satisfy $grg^{-1} \in H\setminus C_{\mathrm{ns}}(p^k)$. In this case, any solution satisfies the equation \eqref{bla2}, which we write again for clarity reasons: $$\epsilon b_0 z^2=b_0 v^2-2a_0v+\epsilon b_0.$$

\noindent \textbf{Case II.b.(i)} All solutions $g$ also satisfy $gR_1g^{-1}\subseteq C_{\mathrm{ns}}(p^k)$. 

Let $\beta=\min_{g \in R_1}v_p(b)$ as before. If $\beta=0$, we obtain only the trivial solution, as before. Otherwise, we need to count the number of solutions to \eqref{bla2} which also satisfy $p^{k-\beta} \mid z^2-1$ and $p^{k-\beta} \mid v$.

We start by solving the special case $\beta=k$. In that case, we only need to count the number of spolutions to \eqref{bla2}. 

If $b_0$ is not invertible, then $a_0$ is invertible. Note that then $v=0$ is a solution modulo $p$, and Hensel's lemma implies that this solution has a unique lift, for any value of $z^2$. There are $\frac{p^{k-1}(p-1)}{2}$ choices for $z^2$, so this is the number of solutions.

If $b_0$ is invertible, then using the same arguments as before, there are $p+1$ solutions $(z,w)$ mod $p$. Of these solutions, there are either $0$ or $2$ solutions with $z=0$ which we need to exclude, depending on whether $4a_0^2-4\epsilon b_0$ is a square or not. If $4a_0^2-4\epsilon b_0$ is a square, then we obtain $\frac{(p-1)p^{k-1}}{2}$ solutions as before. If it is not a square, we obtain $\frac{(p+1)p^{k-1}}{2}$ solutions.

Now suppose $0<\beta<k$. Note that there is then only one solution mod $p$, which is $z^2=1$, $v=0$. 

If $b_0$ is not invertible, then again by Hensel's lemma, for any choice of $z$ such that $z^2-1$ is divisible by $p^{k-\beta}$, there is a unique lift of the solution $v=0$ to a solution modulo $p^k$. It is easy to see from the equation that this lift satisfies $p^{k-\beta} \mid v$. Thus, in this case we get $p^{\beta}$ solutions, one for each choice of $z^2$.

If $b_0$ is invertible, then for any choice of $v$ such that $p^{k-\beta} \mid v$, there is a unique lift of the solution $z^2=1$ to a solution modulo $p^k$. Again, it is easy to see from the equation that this lift satisfies $p^{k-\beta} \mid z^2-1$. Again, we obtain $p^\beta$ solutions.

\noindent \textbf{Case II.b.(ii)} Every matrix $\sm{a & \epsilon b \\ b & a}\in R_1$ satisfies $a=0$ or $b=0$, and there exists a matrix with $a=0$. In this case, any nontrivial solution $g$ satisfies $g \sm{0 & \epsilon b \\ b & 0}g^{-1}\in H\setminus C_{\mathrm{ns}}(p^k)$, which gives us the equality
$\epsilon z^2+\epsilon =v^2$. This allows us to eliminate $z^2$ from \eqref{bla2}, and obtain $$-2\epsilon b_0=b_0 v^2-2a_0v.$$

If $b_0$ is invertible, the equation can be rewritten as $$v^2-2\frac{a_0}{b_0}v+2\epsilon=0.$$

This equation has a double root modulo $p$ if and only if $\frac{a_0^2}{b_0^2}\equiv 2\epsilon \pmod p$. If $2$ is a square modulo $p$, this is impossible and we only get the trivial solution. 

If $2$ is not a square modulo $p$, then this is possible. In that case, $v\equiv \frac{a_0}{b_0} \pmod p$ and $\frac{a_0^2}{b_0^2}\equiv 2\epsilon \pmod p$. Note that this implies $z^2\equiv 1 \pmod p$, so for any value of $v$ that satisfies the conditions, there is a unique suitable value of $z$.

It suffices to count the number of solutions to $$v^2-2\frac{a_0}{b_0}v+\frac{a_0^2}{b_0^2}-px=0,$$ where $p x=\frac{a_0^2}{b_0^2}-2\epsilon$. If $s=v-\frac{a_0}{b_0}$, this is equivalent with $$s^2=px.$$
There may be $0$ solutions to this equation. Suppose that there is at least $1$ solution. 

If $px=0$, there are $p^{\lfloor k/2 \rfloor}$ solutions. Adding $1$ from the trivial solution, we obtain $p^{\lfloor k/2 \rfloor}+1$. 

Otherwise, let $px=p^{2t}c^2$ where $2t<k$ and $c$ invertible. Then $v_p(s)=t$, and $s=p^t r$ for some invertible $r$, and $r^2\equiv c^2 \pmod p^{k-2t}$. There are $2$ choices for $r \pmod {p^{k-2t}}$, hence $2p^{t}$ choices mod $p^{k-t}$. In total, there are $2p^{t}$ choices for $v$. Adding $1$ from the trivial solution, we obtain $2p^t+1$.

If the equation has only single roots, then its discriminant needs to be a square modulo $p$ to have solutions. That is, $\frac{a_0^2}{b_0^2}-2\epsilon$ needs to be a square. As before, this can be achieved for some $a_0$ and $b_0$.  In that case, there are exactly $2$ solutions $v_1$ and $v_2$ by Hensel's lemma. It remains to check whether $z^2=\frac{v_i^2}{\epsilon}-1$ is a square. Depending on whether none, one or both of them is a square, we obtain $1$, $2$ or $3$ solutions respectively. We claim that we can always choose $\frac{a_0}{b_0}$ for each of the possibilities when $p$ is large enough. 

Let $2\frac{a_0}{b_0}=c$. Then $c^2-8\epsilon$ is a square. Let $d^2=c^2-8\epsilon$. Then $\frac{v_1^2}{\epsilon}-1=\frac{(d+c)^2}{4\epsilon}-1$ and $\frac{v_2^2}{\epsilon}-1=\frac{(d-c)^2}{4\epsilon}-1.$

Multiplying everything by suitable squares and setting $c-d=A$ and $c+d=\frac{8\epsilon }{A}$, we obtain the following two expressions: $16\epsilon-A^2$ and $\epsilon A^2-4\epsilon^2$. 

\begin{lemma}
    Let $p>11$ be a prime, and let $f(x)$ and $g(x)$ be irreducible quadratic polynomials in $\F_p[x]$ with no common factors. Then there is some $A\in \F_p$ such that $f(A)$ and $g(A)$ are both squares in $\F_p^\times$.
\end{lemma}
\begin{proof}
Write \(\chi\) for the quadratic character modulo \(p\).

For signs \(s,t\in\{+1,-1\}\) let
\[
N_{s,t}=\#\{A\in\F_p:\ \chi(f(A))=s,\ \chi(g(A))=t\}.
\]
Using the standard identity \(\mathbf{1}_{\chi(f(x))=s}=(1+s\chi(f(x)))/2\) (and similarly for \(\chi_g\)) we obtain
\[
N_{s,t}=\sum_{A\in\F_p}\frac{1+s\chi(f(A))}{2}\cdot\frac{1+t\chi(g(A))}{2}
=\frac{1}{4}\Big(p+s\Sigma_f+t\Sigma_g+st\Sigma_{fg}\Big),
\]
where
\[
\Sigma_f=\sum_{A\in\F_p}\chi(f(A)),\qquad
\Sigma_g=\sum_{A\in\F_p}\chi(g(A)),\qquad
\Sigma_{fg}=\sum_{A\in\F_p}\chi(f(A)g(A)).
\]

For an irreducible quadratic polynomial \(q(x)=ax^2+bx+c\) one has \(\sum_{x\in\F_p}\chi(q(x))=-\chi(a).\) In particular, \(|\Sigma_f|=|\Sigma_g|=1\).  
The product \(f(x)g(x)\) is a quartic polynomial with no repeated roots, so the Weil bound gives us
\[
|\Sigma_{fg}|\leq  3\sqrt p.
\]

Combining these estimates we obtain, for every sign pair \((s,t)\),
\[
N_{s,t}\ge \frac{1}{4}\Big(p - \big(|\Sigma_f|+|\Sigma_g|+|\Sigma_{fg}|\big)\Big)
\ge \frac{1}{4}\big(p - 2 - 3\sqrt p\big).
\]
The right-hand side is positive whenever \(p-2-3\sqrt p>0\), which holds for every prime \(p\ge 13\). Hence for \(p\ge 13\) each \(N_{s,t}\) is positive. In particular, there exists \(A\) with \(\chi_f(A)=\chi_g(A)=+1\), i.e.\ both \(f(A)\) and \(g(A)\) are nonzero squares.
\end{proof}

Applying this for polynomials $16\epsilon-x^2$ and $\epsilon x^2-4\epsilon^2$, we obtain that each combination of squares and non-squares occurs. It remains to check primes up to $11$. A direct check shows that for $p=3$ either both polynomials are squares or non-squares, while for $p=5,7,11$ there can be $0,1$ or $2$ squares. In conclusion, adding the trivial solution, we obtain $1$, $2$ or $3$ solutions, unless $p=3$ in which case we obtain $1$ or $3$ solutions.

Now suppose that $b_0$ is not invertible. Then $a_0$ is invertible and $v$ is divisible by $p$. By Hensel's lemma, since the derivative of $b_0v^2-2a_0v+2\epsilon b_0$ at $v=0$ is nonzero, it follows that there is a unique solution $v$, and $v\equiv 0 \pmod p$. It follows that $\epsilon z^2+\epsilon\equiv 0 \pmod p$, so there is $1$ nontrivial solution if $-1$ is a square, and $0$ nontrivial solutions otherwise. In any case, we obtain either $1$ or $2$.
\end{proof}

\section{Number of rational CM points on a modular curve}\label{sec5}
The goal of this section is to provide an algorithm for the number of rational CM points on a given modular curve $X_H$. We start by recalling known results about images of Galois representations of elliptic curves with complex multiplication, due to Lozano-Robledo \cite{alvarocm}. In what follows, $K$ is an imaginary quadratic field, $\mathcal O_{K,f}$ is the order of $K$ of conductor $f$, and $j_{K,f}$ is any value of $j$-invariant of an elliptic curve with CM by $\mathcal O_{K,f}$. Note that for rational $j$-invariants, $j_{K,f}$ is uniquely determined.

\begin{theorem}\cite[Theorem 1.1]{alvarocm}\label{alvaro1}
    Let $E/\Q(j_{K,f})$ be an elliptic curve with CM by $\mathcal{O}_{K,f}$, let $N\ge 3$, and let
\[
\rho_{E,N}:\Gal(\overline\Q/\Q(j_{K,f}))\to \Aut(E[N])\cong \GL(2,\Z/N\Z).
\]
We define groups of $\GL_2(\Z/N\Z)$ as follows:
\begin{itemize}
    \item If $\Delta_K f^2 \equiv 0 \pmod 4$, let $\delta=\Delta_K f^2/4$ and $\phi=0$.
    \item If $\Delta_K f^2 \equiv 1 \pmod 4$, let $\delta=(\Delta_K-1)f^2/4$ and $\phi=f$.
\end{itemize}
Define the Cartan subgroup $\mathcal C_{\delta,\phi}(N)\subset\GL_2(\Z/N\Z)$ by
\[
\mathcal C_{\delta,\phi}(N)=
\left\{
\begin{pmatrix}
a+b\phi & b\\[4pt]
\delta b & a
\end{pmatrix}
: a,b\in\Z/N\Z,\ a^2+ab\phi-\delta b^2\in(\Z/N\Z)^\times
\right\}.
\]
Let
\[
\mathcal N_{\delta,\phi}(N)=\langle \mathcal C_{\delta,\phi}(N),\ c_\phi\rangle,\qquad
c_\phi=\begin{pmatrix}-1&0\\[4pt]\phi&1\end{pmatrix}.
\]
Then there is a $\Z/N\Z$-basis of $E[N]$ such that the image of $\rho_{E,N}$ is contained in $\mathcal N_{\delta,\phi}(N)$. Moreover,
\begin{enumerate}
    \item $\mathcal C_{\delta,\phi}(N)\cong (\mathcal{O}_{K,f}/N\mathcal{O}_{K,f})^\times$ is a subgroup of index $2$ in $\mathcal N_{\delta,\phi}(N)$, and
    \item the index of the image of $\rho_{E,N}$ in $\mathcal N_{\delta,\phi}(N)$ coincides with the order of the Galois group
    $\Gal(K(j_{K,f},E[N])/K(j_{K,f},h(E[N])))$ for a Weber function $h$, and it divides the order of
    $\mathcal{O}_{K,f}^\times/\mathcal{O}_{K,f,N}^\times$, where
    $\mathcal{O}_{K,f,N}^\times=\{u\in\mathcal{O}_{K,f}^\times: u\equiv 1\pmod{N\mathcal{O}_{K,f}}\}$.
\end{enumerate}
\end{theorem}

\begin{remark}
    Note that the choices of the parameters $\delta$ and $\phi$ (and consequently, the group $\mathcal C_{\delta, \phi}$) in our statement of the theorem differs from the statement in the original paper by Lozano-Robledo. However, those are conjugate subgroups of $\GL_2(\Z/N\Z)$, so the claim still holds. In fact, our choice of parameters was used in the theorem that follows.
\end{remark}

\begin{theorem}\cite[Theorem 1.2.(1)]{alvarocm} \label{alvaro2}
    Using notation from \Cref{alvaro1}, for every $K$ and $f\geq 1$, and a fixed $N\geq 3$, there is an elliptic curve $E/\Q(j_{K,f})$ with CM by $\mathcal O_{K,f}$ such that the image of $\rho_{E,N}$ is precisely $\mathcal N_{\delta,\phi}(N)$.
\end{theorem}

These two theorems tell us what the Galois image looks like in some basis. In order to apply results from Section 2, we must also know what the automorphisms of $E$ look like in this basis. From the proof of \cite[Theorem 1.1]{alvarocm}, it follows that the basis in which the image of the Galois representation is of the above format is the one corresponding to $(f\tau, 1)$, where $\tau=\frac{\sqrt{\Delta_K}}{2}$ or $\tau=\frac{1+\sqrt{\Delta_K}}{2}$, depending on $\Delta_K \pmod 4$. One has $\mathcal{O}_{K,f}=\Z+f\Z[\tau]$. It is now easy to deduce the action of automorphisms. Let $A$ denote the image of the action of $\Aut(E)$ in this basis.

\begin{itemize}
    \item $j=1728$. There are $4$ automorphisms, corresponding to multiplication by powers of $i$. One has $f=1$ and $\tau=i$. Moreover, the action of multiplication by $i$ in the basis $(i, 1)$ is given by the matrix $\sm{0 & 1 \\-1 & 0}$, so $A=\langle \sm{0 & 1 \\-1 & 0} \rangle $. 
    \item $j=0$. There are $6$ automorphisms, corresponding to multiplication by powers of $\frac{-1+\sqrt{-3}}{2}$. One has $f=1$ and $\tau=\frac{1+\sqrt{-3}}{2}$. Moreover, the action of multiplication by $\tau$ in the basis $(\tau, 1)$ is given by the matrix $\sm{1 & 1\\ -1 & 0}$.
    \item $j\not \in \{0, 1728\}$. One simply has $A=\{\pm I\}$.
\end{itemize}

Fix a rational CM $j$-invariant. One can now easily deduce an algorithm for the number of rational points $x$ on a modular curve $X_H$ such that $j(x)=j_0$.

Suppose that the input is the $j$-invariant with the corresponding quadratic order which has minimal discriminant $\Delta_K$ and conductor $f$, the level $N$ and a group $H \leq \GL_2(\Z/N\Z)$ which contains $-I$ (if it does not contain $-I$, we consider $\langle H, -I\rangle$ instead).

\begin{itemize}
    \item Calculate the conductor $f$ and the discriminant $\Delta_K$ of the corresponding quadratic order $\mathcal O_{K,f}.$ 
    \item If $\Delta_Kf^2\equiv 0 \pmod 4$, set $\delta=\Delta_Kf^2/4$ and $\phi=0$. Else, set $\delta=\frac{\Delta_K-1}{4}f^2$ and $\phi=f$.
    \item Let $$\mathcal C_{\delta, \phi}(N)=\left\{ \sm{a+b\phi & b \\ \delta b & a} : a,b \in \Z/N\Z, \ a^2+ab\phi-\delta b^2 \in (\Z/N\Z)^\times \right\}$$ and $$\mathcal N_{\delta, \phi}(N)=\left\langle \mathcal C_{\delta, \phi}, \sm{-1 & 0 \\ \phi & 1}\right\rangle.$$ 
    \item Then there exists an elliptic curve with the given $j$-invariant such that the mod $N$ Galois image $R$ is equal to $\mathcal N_{\delta, \phi}(N)$ in some basis, and the image of the action of automorphisms is $A$.
    \item If $j_0$ is $0$ or $1728$, calculate the number of double cosets $HgA$ with $gR \subseteq HgA$. 
    \item Otherwise, calculate the number of cosets $Hg$ with $gRg^{-1} \subseteq H$.
\end{itemize}

\begin{remark}
    In \cite{rousemayle}, Mayle and Rouse give an algorithm which determines all rational points on a modular curve which has a map to an elliptic curve of rank $0$. In one step of their calculations, they use an optimized version of our algorithm, which can be found on their GitHub repository \cite{rousegithub}. They optimize our code in several ways. One improvement is to use Sutherland's functions from the repository \cite{Sutherland2025Magma} to build the Cartan subgroups and their normalizers instead of building them directly.

    Another improvement is in the cases $j=0$ and $j=1728$. The condition that one needs to check is that for every double coset $HgA$, one has $gRg^{-1}\subseteq HgAg^{-1}$, where $R=\mathcal N_{\delta, \phi}(N)$ and $A$ is the group of automorphisms. Let $g$ be a double coset representative. Let $K=gRg^{-1}\cap H$. Let $g_1, \ldots, g_r$ be the right coset representatives of $K$ inside $gRg^{-1}$. One needs to check, for each $i$, whether $$Kg_i \subseteq HgAg^{-1}.$$
    Since $K \subseteq H$, this is equivalent to $g_i\in HgAg^{-1}$. Furthermore, note that $HgAg^{-1}$ is the union of two or three right cosets of $H$, so this condition is easy to check.
\end{remark}

\begin{remark}
The described algorithm works for rational CM points on modular curves. It is often of interest to determine $K$-rational points on $X_H$ above a fixed  $j$-invariant $j_0\in \Q$ for some larger number field $K$. In that case, $\rho_{E,N}(G_K)$ is a subgroup of $\mathcal N_{\delta, \phi}(N)$, and in general we do not know which subgroup it is. However, if $K$ is exactly the quadratic CM field of $j_0$, then one can choose $E$ with  $j(E)=j_0$ so that $\rho_{E,N}(G_K)$ equals $\mathcal C_{\delta, \phi}(N)$ (this follows from the proof of \Cref{alvaro1}) and one can apply the same algorithm, just replacing $\mathcal N_{\delta, \phi}(N)$ with $\mathcal C_{\delta, \phi}(N)$. This is often useful, since quadratic CM points on modular curves are often defined exactly over CM fields of rational elliptic curves.
\end{remark}

\begin{example}
    In \cite[Example 4.2]{filip-ja}, the number of quadratic points on $X_0(88)$ above $j$-invariants $-3375$ and $16581375$ is determined.  The corresponding imaginary quadratic orders are $\Z\left[\frac{1+\sqrt{-7}}{2}\right]$ and $\Z[\sqrt {-7}]$. In that article, the number of points was found using properties of isogenies of CM curves and the splitting behaviour of prime $2$ and its powers inside the quadratic orders. In the special case of $X_0(n)$, the point count can be done using these tricks, but for general modular curves $X_H$ this may not work.

    Using our code, we quickly find that there are $8$ points on $X_0(88)$ above each of the mentioned $j$-invariants over their CM field $\Q(\sqrt{-7})$. Magma's \texttt{time} function returns 0.440 and 0.660s of CPU time for each of the $j$-invariants. For comparison, the code used in \cite{filip-ja} as a sanity check, based on methods developed in \cite{nikola}, takes much longer. There, the method is to compute the $j$-map $X_0(88)\to \PP^1$ and then calculate the preimages.  
\end{example}

\bibliographystyle{siam}
\bibliography{lit}

\end{document}